\documentclass[11pt]{amsart}
\usepackage{amsmath,amsthm,amscd,amsfonts,amssymb,epic,eepic,bbm, graphicx, youngtab,txfonts, ytableau}
\usepackage{tikz,color,graphicx,tikz-cd,xcolor,tkz-euclide}
\usepackage{hyperref}
\usepackage{cleveref}
\numberwithin{equation}{section}
\usepackage{bm}
\allowdisplaybreaks

\newtheorem{thm}{Theorem}[section]

\newtheorem{lem}[thm]{Lemma}

\newtheorem{prop}[thm]{Proposition}

\theoremstyle{definition}
\newtheorem{defn}[thm]{Definition}
\newtheorem{quest}[thm]{Question}

\theoremstyle{remark}
\newtheorem{rem}[thm]{Remark}

\newcommand\MUD{\operatorname{MUD}}

\newcommand\UD{\operatorname{UD}}

\newcommand\Cat{\operatorname{Cat}}

\newcommand\VT{\operatorname{VT}}
\newcommand\MVT{\operatorname{MVT}}

\newcommand\Par{\operatorname{Par}}

\newcommand{\la}{\lambda}

\newcommand{\SYT}{\operatorname{SYT}}

\newcommand{\pop}{{p_1^{\perp}}}

\newcommand\QQ{\mathbb{Q}}

\newcommand{\ZZ}{\mathbb{Z}}

\renewcommand\vec[1]{\mathbf{#1}}
\newcommand\vx{\mathbf{x}}

\newcommand\vu{\mathbf{u}}
\newcommand\vv{\mathbf{v}}

\DeclareMathOperator*{\Pf}{Pf}

\newcommand\flr[1]{\lfloor  #1 \rfloor}

\newcommand\sorth{\mathrm{so}}
\newcommand\orth{\mathrm{o}}

\def\al{\alpha}

\def\ep{\varepsilon}

\def\Z{\mathbb Z}
\def\Zp{\mathbb Z_{\geq0}}
\def\sgn{\operatorname{sgn}}

\begin{document}

\title{Bounded Littlewood identities with
fixed number of odd rows or odd columns}

\author{JiSun Huh}
\address{Department of Mathematics, University of Seoul, Seoul, Republic of Korea}
\email{hyunyjia@yonsei.ac.kr}

\author{Jang Soo Kim}
\address{Department of Mathematics, Sungkyunkwan University, Suwon, Republic of Korea}
\email{jangsookim@skku.edu}

\author{Christian Krattenthaler}
\address{Fakult\"at f\"ur Mathematik, Universit\"at Wien,
Oskar-Morgenstern-Platz~1, A-1090 Vienna, Austria.}
\email{Christian.Krattenthaler@univie.ac.at}

\author{Soichi Okada}
\address{%
Graduate School of Mathematics, Nagoya University,
Furo-cho, Chikusa-ku, Nagoya 464-8602, Japan}
\email{okada@math.nagoya-u.ac.jp}


\begin{abstract} 
A Littlewood identity is an identity equating a sum of Schur functions
with an infinite product. A bounded Littlewood identity is one where
the sum is taken over the partitions with a bounded number of
rows or columns. The price to pay is that the infinite product has to
be replaced by a determinant.
  The focus of this article is on refinements of such bounded Littlewood
identities where one also prescribes the number of odd-length rows or columns
of the partitions. Goulden [{\it Discrete Math.} {\bf99} (1992), 69--77]
had given such a refinement in which the number of columns is bounded
and the number of odd-length rows is prescribed. We provide
refinements where the number of columns is bounded and the number of
odd-length columns is prescribed. Furthermore, we present new formulations
of such bounded Littlewood identities involving skewing operators.
As corollaries we obtain non-standard formulas for numbers of standard
Young tableaux with restricted shapes as above.
In the last part of the article we discuss combinatorial
interpretations of such identities in terms of up-down tableaux.
As corollaries, we obtain identities between numbers of standard Young
tableaux and numbers of (marked) vacillating tableaux.
\end{abstract}

\maketitle
 \setcounter{tocdepth}{1}
\tableofcontents



\section{Introduction}
\label{sec:intro}

It is common understanding nowadays, when one speaks of a {\it``Littlewood
identity"}, to mean an identity that equates an infinite sum of {\it Schur
functions} to an infinite product. (We refer the reader to
Section~\ref{sec:pre} for all definitions, and to \cite{Macdonald} and
\cite[Chapter~7]{StanBI} for in-depth introductions to the theory of
symmetric functions.) The prototypical example is
$$
  \sum_{\lambda} s_{\lambda}(\vx) =
  \frac{ 1 } { \prod_i ( 1 - x_i  ) \prod_{i<j} ( 1 - x_i x_j ) },
$$
where $s_\lambda(\vx)$ denotes the Schur function indexed by the
partition~$\lambda$, even though this particular identity is due
to Schur. (The reader may want to consult the introduction
of~\cite{HKKO} for a more detailed discussion of the history of
Littlewood(-type) identities.) Other Littlewood identities provide
product formulas for sums of Schur functions over subsets of the set
of partitions, such as partitions with only even parts or
self-conjugate partitions. A {\it bounded Littlewood identity} is one
where the sum of Schur functions is over partitions with a restricted
number of columns or rows. The theorem below contains the two
classical bounded Littlewood identities. We refer the reader again to
the introduction of~\cite{HKKO} for attribution and slightly
convoluted history of these identities.

\begin{thm}[\sc Two bounded Littlewood identities]
  \label{thm:BK-intro}
  For a nonnegative integer \( w \), we have
  \begin{equation}
    \label{eq:BK_odd1}
    \sum_{\la:\lambda_1\le 2w+1} s_{\lambda}(\vx)
    =
    \sum_{k \ge 0} e_k(\vx) \cdot
    \det_{1 \le i,j \le w}
    \left(
      f_{i-j}(\vx) - f_{i+j}(\vx)
    \right)
\end{equation}
and
\begin{equation}
    \label{eq:BK_even1}
    \sum_{\la:\lambda_1\le 2w} s_{\lambda}(\vx)
    =
    \det_{1 \le i,j \le w}
    \left(
      f_{i-j}(\vx) + f_{i+j-1}(\vx)
    \right),
  \end{equation}
  where \( e_k(\vx) \) is the \( k \)-th elementary symmetric function and
  \begin{equation}\label{eq:f_k}
    f_r(\vx)=\sum_{n\in \Z}e_n(\vx) e_{n+r}(\vx).
  \end{equation}
\end{thm}

The focus of the present article is on {\it refinements} of the above
bounded Littlewood identities and their combinatorics.
In these refinements the number of
odd rows or odd columns of the partitions over which the summation
of Schur functions is taken is fixed. Here, an \emph{odd row}
 (respectively \emph{odd column}) means a row (respectively column) of odd length.
Indeed, there is already a
substantial literature on such refinements. It is our goal
(1) to provide an overview of the relevant results that are scattered
over the literature, (2) to exhibit the connections between the
various existing results, and
(3) to present new results, in particular
combinatorial interpretations of the right-hand sides of such
refined bounded Littlewood identities in terms of up-down tableaux.

\medskip
In the remainder of this introduction, we provide an overview over
the material that we present in this article. As already indicated,
all relevant definitions are given in Section~\ref{sec:pre}.

\medskip
{\sc Overview --- Section~\ref{sec:rowGoulden}: column refinements of
Theorem~\ref{thm:BK-intro}.}
To begin with, we
recall Goulden's (row) refinements~\cite{Goulden1992}
of the two bounded Littlewood identities
in Theorem~\ref{thm:BK-intro}.
In order to state these,
let \( r(\lambda) \) (respectively~\( c(\lambda) \)) denote the number of odd
rows (respectively~odd columns) of \(\lambda \). 

\begin{thm}[\sc {\cite[\sc Theorems~2.4 and~2.6]{Goulden1992}}]\label{thm:Goulden}
  For integers \( w\geq 1 \) and \( k \geq 0  \), we have
  \begin{equation}\label{eq:Goulden2m+1,k}
\underset{r(\lambda)=k}
    {\sum_{\lambda:\lambda_1\le 2w+1}} s_{\lambda}(\vx)
    =e_k(\vx) \det_{1\le i,j\le w} \left( f_{i-j}(\vx)-f_{i+j}(\vx) \right)
   \end{equation}
   and
   \begin{equation} 
    \label{eq:Goulden2m,k}
\underset{r(\lambda)=k}
    {\sum_{\lambda:\lambda_1 \le 2w}} s_{\lambda}(\vx)
    =\det_{1\le i,j\le w}
      \left(
       \begin{cases}
        f_{i-j}(\vx)-f_{i+j}(\vx), &\mbox{if \( 1 \le i \le w-1 \)} \\
        f_{i-j+k}(\vx)-f_{i+j+k}(\vx), &\mbox{if \( i=w \)}
       \end{cases}
      \right).
   \end{equation}
  In particular, 
  \begin{equation}
    \label{eq:Goulden2m,0}
\underset{r(\lambda)=0}
    {\sum_{\lambda:\lambda_1\le 2w+1}} s_{\lambda}(\vx)
    =\underset{r(\lambda)=0}
    {\sum_{\lambda:\lambda_1\le 2w}} s_{\lambda}(\vx)
    =\det_{1\le i,j\le w} \left( f_{i-j}(\vx)-f_{i+j}(\vx) \right),
  \end{equation}
with $f_r(\vx)$ as given in \eqref{eq:f_k}.
\end{thm}

We present here a refinement of Theorem~\ref{thm:BK-intro} in which
one keeps track of the number of odd {\it columns} of the partitions
over which the sum is taken.
For convenience, we put
\[
e(\vx) = \sum_{n \ge 0} e_n(\vx),
\qquad
\overline{e}(\vx) = \sum_{n \ge 0} (-1)^n e_n(\vx).
\]

\begin{thm}
\label{thm:row_Goulden}
Let \( w \) be a nonnegative integer,
and let \( u \) be an indeterminate.
\begin{enumerate}
\item[(1)]
We have
\begin{multline}\label{eq:even+}
\sum_{\lambda: \lambda_1 \le 2w}
 \left( u^{c(\lambda)} + u^{2w-c(\lambda)} \right) s_{\lambda}(\vx)
\\
 =
 \frac{1}{2}
 \sum_{k=0}^w
 \left( u^k + u^{2w-k} \right)
 \det_{1 \le i, j \le w} \left(
  \begin{cases}
   f_{i-j}(\vx) + f_{i+j-2}(\vx), &\mbox{if \(1 \le i \le w-k \)} \\
   f_{i-j+1}(\vx) + f_{i+j-1}(\vx), &\mbox{if \( w-k+1 \le i \le w \)}
  \end{cases}
 \right)
\end{multline}
and
\begin{multline}\label{eq:even-}
\sum_{\lambda:\lambda_1 \le 2w}
 \left( u^{c(\lambda)} - u^{2w-c(\lambda)} \right) s_{\lambda}(\vx)
\\
 =
e(\vx) \overline{e}(\vx) \cdot 
\sum_{k=0}^{w-1}
 \left( u^k - u^{2w-k} \right)
 \det_{1 \le i, j \le w-1} \left(
  \begin{cases}
   f_{i-j}(\vx) - f_{i+j}(\vx), &\mbox{if \( 1 \le i \le w-k-1 \)} \\
   f_{i-j+1}(\vx) - f_{i+j+1}(\vx), &\mbox{if \( w-k \le i \le w-1 \)}
  \end{cases}
 \right).
\end{multline}
\item[(2)]
We have
\begin{multline}\label{eq:odd+}
\sum_{\lambda:\lambda_1 \le 2w+1}
 \left( u^{c(\lambda)} + u^{2w+1-c(\lambda)} \right) s_{\lambda}(\vx)
\\
 =
e(\vx) \cdot
\sum_{k=0}^w
 \left( u^k + u^{2w+1-k} \right)
 \det_{1 \le i, j \le w} \left(
  \begin{cases}
   f_{i-j}(\vx) - f_{i+j-1}(\vx), &\mbox{if \( 1 \le i \le w-k \)} \\
   f_{i-j+1}(\vx) - f_{i+j}(\vx), &\mbox{if \( w-k+1 \le i \le w \)}
  \end{cases}
 \right)
\end{multline}
and
\begin{multline}\label{eq:odd-}
\sum_{\lambda:\lambda_1 \le 2w+1}
 \left( u^{c(\lambda)} - u^{2w+1-c(\lambda)} \right) s_{\lambda}(\vx)
\\
 =
\overline{e}(\vx) \cdot 
\sum_{k=0}^w
 \left( u^k - u^{2w+1-k} \right)
 \det_{1 \le i, j \le w} \left(
  \begin{cases}
   f_{i-j}(\vx) + f_{i+j-1}(\vx), &\mbox{if \( 1 \le i \le w-k \)} \\
   f_{i-j+1}(\vx) + f_{i+j}(\vx), &\mbox{if \( w-k+1 \le i \le w \)}
  \end{cases}
 \right).
\end{multline}
\end{enumerate}
Again, the series $f_r(\vx)$ are given in \eqref{eq:f_k}.
\end{thm}

If we put $u=1$ in \eqref{eq:even+} and~\eqref{eq:odd+} and use Lemma~\ref{lem:sum_det} below 
and column operations, we recover the bounded Littlewood identities~\eqref{eq:BK_even1} and~\eqref{eq:BK_odd1}, respectively.
(If we put $u=1$ in~\eqref{eq:even-} and~\eqref{eq:odd-}, then both sides become zero.)
We note that the $u=0$ cases of Theorem~\ref{thm:row_Goulden} were first proved in 
\cite[Theorem~2.3(3)]{Okada1998} 
as classical group character identities of rectangular shape 
(see also Section~\ref{sec:nearlyrect}).
Cylindric versions were given in the authors' previous paper~\cite{HKKO}.

We prove Theorem~\ref{thm:row_Goulden} in
Section~\ref{sec:rowGoulden}. Our main tool is the {\it minor
  summation formula} of Ishikawa and
Wakayama~\cite[Theorem~1]{IsWaAA}; see Theorem~\ref{thm:IsWa}.

\medskip
{\sc Overview --- Section~\ref{sec:nearlyrect}: equivalence of
  Theorem~\ref{thm:row_Goulden} and identities 
  for classical group characters of nearly rectangular shape.}
As it turns out, the refined bounded Littlewood identities in
Theorem~\ref{thm:row_Goulden} are equivalent to identities for
classical group characters obtained by the third author
in~\cite{Krattenthaler1998}. We explain this (non-trivial) equivalence
in Section~\ref{sec:nearlyrect}.

\medskip
{\sc Overview --- Section~\ref{sec:skew}: skewing operators.}
Here we present new formulations of the bounded
Littlewood identities in Theorem~\ref{thm:BK-intro} 
and their refinements in Theorems~\ref{thm:Goulden} and \ref{thm:row_Goulden}
in terms of {\it skewing operators}. In order to state these, we recall that,
given a symmetric function $f(\vx)$, the associated skewing operator
$f^\perp$ is, by definition, the adjoint of multiplication by $f(\vx)$
with respect to the {\it Hall inner product}, that is,
$\langle f^\perp r_1(\vx),r_2(\vx) \rangle =\langle
r_1(\vx),f(\vx)r_2(\vx) \rangle,$   
for all symmetric functions \( r_1(\vx)\) and \(r_2(\vx)\).
We refer the reader again to Section~\ref{sec:pre} for full definitions.
In the theorems below, we only need the skewing operator associated
with the first power-sum symmetric function $p_1(\vx)=x_1+x_2+\cdots$.

\begin{thm}\label{thm:BK2}
  For a nonnegative integer \( w \),
  \begin{align}
  \label{eq:oddpop}  
     \sum_{\la:\lambda_1\le 2w+1} s_{\lambda}(\vx)
    &= e(\vx) \det_{1\le i,j \leq w}\left((\pop)^{i+j-2} (f_{0}(\vx)-f_{2}(\vx))\right),\\
  \label{eq:evenpop}
     \sum_{\la:\lambda_1\le 2w} s_{\lambda}(\vx) 
     &= \det_{1\le i,j \leq w}\left((\pop)^{i+j-2}(f_0(\vx)+f_1(\vx))\right).
  \end{align} 
\end{thm}

We prove as well analogous formulations for the refinements in
Theorems~\ref{thm:Goulden} and \ref{thm:row_Goulden}.

\begin{thm} \label{thm:Goulden2}
  For nonnegative integers \( w \) and \( k \),
 \begin{align}
 \label{eq:oddref_k}
\underset{r(\lambda)=k}
 {\sum_{\lambda:\lambda_1\le 2w+1}}s_{\lambda}(\vx)
   &=e_k(\vx)\det_{1\le i,j \leq w}\left( (\pop)^{i+j-2} (f_{0}(\vx)-f_{2}(\vx)) \right),\\
  \label{eq:evenref_k}
\underset{r(\lambda)=k}
  {\sum_{\lambda:\lambda_1\le 2w}}s_{\lambda}(\vx)
   &=\det_{1\le i,j\le w} 
   \left( (\pop)^{j-1} (f_{i+k\delta_{i,w}-1}(\vx)-f_{i+k\delta_{i,w}+1}(\vx)) \right),
  \end{align}  
  where \( \delta_{i,j}=1 \) if \( i=j \) and \( \delta_{i,j}=0 \) otherwise.
\end{thm}

\begin{thm}\label{thm:row_Goulden2-2pop}
For nonnegative integers \( w \) and \( k \) such that \( 0\le k\le w \),
we have
\begin{align}
  \label{eq:schur-odd-kpop}
\underset{c(\lambda) = k}
  {\sum_{\lambda:\lambda_1 \le 2w+1}} s_{\lambda}(\vx)
  +\underset {c(\lambda) = 2w+1-k}
  {\sum_{\lambda:\lambda_1 \le 2w+1}} s_{\lambda}(\vx)
  &= e(\vx) 
   \det_{1 \le i, j \le w}
  \left( (\pop)^{j-1} (f_{i+\chi(i>w-k)-1}(\vx)-f_{i+\chi(i>w-k)}(\vx)) \right),\\
  \label{eq:schur-odd-2h+1-kpop}
  \underset {c(\lambda) = k}
  {\sum_{\lambda:\lambda_1 \le 2w+1}} s_{\lambda}(\vx)
  -\underset {c(\lambda) = 2w+1-k}
  {\sum_{\lambda:\lambda_1 \le 2w+1}} s_{\lambda}(\vx)
  &= \overline{e}(\vx) 
   \det_{1 \le i, j \le w} 
   \left( (\pop)^{j-1} (f_{i+\chi(i>w-k)-1}(\vx)+f_{i+\chi(i>w-k)}(\vx)) \right),\\
  \label{eq:schur-even-kpop}
  \underset {c(\lambda) = k}
  {\sum_{\lambda:\lambda_1 \le 2w}} s_{\lambda}(\vx)
  +\underset {c(\lambda) = 2w-k}
  {\sum_{\lambda:\lambda_1 \le 2w}} s_{\lambda}(\vx)
  &= 2^{\chi(k=w) }
 \det_{1 \le i, j \le w}    
  \left( (\pop)^{j-1}f_{i+\chi(i>w-k)-1}(\vx) \right),
 \end{align}
and, for  \( 0\le k\le w-1 \),
\begin{align}
  \label{eq:schur-even-2h-kpop}
  \underset {c(\lambda) = k}
  {\sum_{\lambda:\lambda_1 \le 2w}} s_{\lambda}(\vx)
  -\underset {c(\lambda) = 2w-k}
  {\sum_{\lambda:\lambda_1 \le 2w}} s_{\lambda}(\vx)
  &= e(\vx) \overline{e}(\vx) 
   \det_{2 \le i, j \le w} 
 \left( (\pop)^{j-2} (f_{i+\chi(i>w-k)-2}(\vx)-f_{i+\chi(i>w-k)}(\vx)) \right),
\end{align}
where $\chi(P)$ is defined to be \( 1 \) if a statement $P$ is true and $0$ otherwise.
\end{thm}

The proofs of these three theorems are found in
Section~\ref{sec:skew}. They are based on elementary properties of the
skewing operator $\pop$ when applied to the series~$f_r(\vx)$ and row
and column operations applied to the determinants in
Theorems~\ref{thm:BK-intro}, \ref{thm:Goulden}, and~\ref{thm:row_Goulden}.

\medskip
{\sc Overview --- Section~\ref{sec:SYT-bounded}:
unusual formulas for numbers of standard Young
tableaux of bounded width.}
In Section~\ref{sec:SYT-bounded}, we show that the formulas in
Theorems~\ref{thm:BK2},
\ref{thm:Goulden2}, and~\ref{thm:row_Goulden2-2pop} can be used to
derive formulas for the number of standard Young tableaux of bounded
width, some of them known but derived only recently in a completely different
manner, some of them new.

Let \( \SYT_{n,w} \) denote the set of standard Young tableaux of size \( n \) with width at most \( w \),
where the \emph{width} is the number of columns.
In order to explain the context, we have to recall Gessel's classical
result~\cite[Eq.~(22), after application of the operator~$\theta$;
  cf.\ p.~277]{Gessel1990}, which, as Gessel points out,
is implicit in papers of Gordon and Houten~\cite{Gordon2,Gordon5} and of
Bender and Knuth~\cite{BK72}. 

\begin{thm} \label{thm:Gessel}
For a nonnegative integer \( w \), 
\begin{align}
  \label{eq:Gessel1}
  \sum_{n\ge0} |\SYT_{n,2w+1}| \frac{x^n}{n!}
  &= \exp(x) \det_{1\le i,j\le w} \left( I_{i-j}(2x)-I_{i+j}(2x) \right),\\
  \label{eq:Gessel2}
  \sum_{n\ge0} |\SYT_{n,2w}| \frac{x^n}{n!}
  &= \det_{1\le i,j\le w} \left(I_{i-j}(2x)+I_{i+j-1}(2x) \right),
  \end{align}
  where 
  \[
  I_\al(2x)=\sum_{r\ge0}\binom{2r+|\al|}{r}\frac {x^{2r+|\al|}} {(2r+|\al|)!}
  \]
  is the modified Bessel function.
\end{thm}

In \cite[Theorem~10.7]{KLO}, the second author, Lee, and Oh found an integral expression for \( |\SYT_{n,w}| \), 
\[
 \int_{O(w+1)}(1-\det(X))(1+{\rm Tr}(X))^n d\mu(X),
\]
and gave an explicit formula by evaluating the integral. 
In the first formula below, $\Cat(n)$ stands for the {\it$n$-th
  Catalan number} if $n$ is a nonnegative integer, that is,
$\Cat(n)=\frac {1} {n+1}\binom {2n}n$, and $\Cat(n)=0$ if $n$ is not a
nonnegative integer.

\begin{thm}[\sc {\cite[Theorem~10.9]{KLO}}] \label{thm:KLO} 
For integers \( w\geq 1 \) and \( n\geq 0 \), we have
\begin{equation}\label{eqn:KLOodd}
  |\SYT_{n,2w+1}|
  =\underset {t_0+t_1+\cdots+t_w=n}
  {\sum_{t_0,t_1,\dots,t_{w}\in \Zp}}
    \binom{n}{t_0,t_1,\dots,t_w}
     \det_{1 \le i,j \le w} \left( \Cat\left(\frac{t_i+2w-i-j}{2}\right) \right)
 \end{equation}
 and
 \begin{equation}\label{eqn:KLOeven}
  |\SYT_{n,2w}|
  =\underset {t_1+\cdots+t_w=n}
  {\sum_{t_1,\dots,t_{w}\in \Zp}}
   \binom{n}{t_1,\dots,t_w}
   \det_{1 \le i,j \le w}
   \left( \ \binom{t_i+2w-i-j}{\lfloor (t_i+2w-i-j)/2\rfloor} \ \right).
  \end{equation}
\end{thm}
We show in Section~\ref{sec:SYT-bounded} that these formulas can
conveniently be derived from the identities in Theorem~\ref{thm:BK2}.
Similarly, from the identities in Theorem~\ref{thm:Goulden2} we may
derive analogous formulas for the number of standard Young tableaux
with bounded width and prescribed number of odd rows.
Let \( f^\lambda \) denote the number of standard Young tableaux of shape \( \lambda \).
In order to state the second of the formulas below,
for nonnegative integers \( r \) and \( s \), define \( F(r,s)\coloneqq \binom{r}{s} -\binom{r}{s-1} \). 
We also define \( F(r,s/2)\coloneqq 0 \) if \( s \) is an odd integer.

\begin{thm}\label{cor:ref_KLO} 
  For nonnegative integers  \( k,w \), and \( n \), we have
  \begin{align}
  \label{eqn:ref_KLOoddSYT}
  \underset {\lambda_1\le 2w+1,\ r(\lambda)=k}
  {\sum_{\lambda\vdash n}} f^{\lambda}
  &=\underset {t_1+\cdots+t_w=n-k}
  {\sum_{t_1,\dots,t_w \in\Zp}}
   \binom{n}{k,t_1,\dots,t_w}
  \det_{1\le i,j\le w} \left(\Cat\left(\frac{t_i+i+j-2}{2}\right)\right),\\
  \label{eqn:ref_KLOevenSYT}
  \underset {\lambda_1\le 2w,\ r(\lambda)=k}
  {\sum_{\lambda\vdash n}} f^{\lambda}
  &=\underset {t_1+\cdots+t_w=n}
  {\sum_{t_1,\dots,t_w \in\Zp}}
   \binom{n}{t_1,\dots,t_w}
  \det_{1\le i,j\le w} 
  \left( F \left( t_i+j-1,\frac{t_i-i-k\delta_{i,w}+j}{2} \right)
   \right).
 \end{align} 
\end{thm} 

The proofs of these identities are also found in
Section~\ref{sec:SYT-bounded}. 

\medskip
{\sc Overview --- Section~\ref{sec:comb-interpr}: up-down tableaux.}
The subject of Section~\ref{sec:comb-interpr} is combinatorial
interpretations of the right-hand sides of the identities in
Theorems~\ref{thm:Goulden} and \ref{thm:row_Goulden}.
Combinatorial interpretations of these identities had been given
before in~\cite[proof of Theorem~3, Eq.~(2.2)]{KratAQ} in terms of
{\it non-crossing two-rowed arrays}. As opposed to that,
our combinatorial interpretations here are in terms of
({\it marked}) {\it up-down tableaux}, see their definition in
Definition~\ref{def:UD} (with refinements in
Definitions~\ref{def:UD2},
\ref{def:UD3}, \ref{def:UD4}, \ref{def:UD5}, and~\ref{def:mud_e}).
Theorems~\ref{thm:Goulden_MUDeven}
and~\ref{thm:Goulden_MUDodd} present our combinatorial interpretations
of the right-hand sides of the identities in
Theorem~\ref{thm:Goulden}.
Theorems~\ref{thm:1} and~\ref{thm:UD-tab} then present our
combinatorial interpretations 
of the right-hand sides of the identities in
Theorem~\ref{thm:row_Goulden}, rewritten in an equivalent form
in Theorem~\ref{thm:row_Goulden2-2}. Our main tool to derive these
results is the main theorem
on {\it nonintersecting lattice paths}, due to
Lindstr\"om~\cite[Lemma~1]{LindAA} (and later rediscovered, among
others, by Gessel and Viennot~\cite{GeViAA,GeViAB}), which allows one
to interpret determinants as the ones on the right-hand sides of the
identities in Theorems~\ref{thm:Goulden} and~\ref{thm:row_Goulden} as
certain generating functions for nonintersecting lattice
paths. We reformulate these families of nonintersecting lattice paths
then in terms of (marked) up-down tableaux.

\medskip
{\sc Overview --- Section~\ref{sec:SYT_walks}: standard Young tableaux and
  lattice walks.}
In the literature, there appear several results on equality of the
number of standard Young tableaux satisfying certain conditions and
the number of certain lattice walks.
We restate here two, due to Zeilberger and to Eu, Fu, Hou, and Hsu,
respectively.

\begin{thm} [\sc Zeilberger \cite{Zeilberger_lazy}]
\label{thm:zeilberger_lazy}
The number  \( |\SYT_{n,2w+1}| \)  is equal to the number of 
  walks of length $n$ in the region 
  $\{(x_1,\dots,x_w):x_1\ge \dots\ge x_w\ge0\}$
  using steps in \( \{ \vec0, \pm \epsilon_1,\dots,\pm \epsilon_w\}\),
  where \( \vec0 = (0,\dots,0) \) and 
  $\epsilon_i=(0,\dots,0,1,0,\dots,0)$, with the $1$ in position~$i$.
\end{thm}

\begin{thm} [\sc Eu et al.~\cite{Eu2013}]
  \label{thm:eu} 
The number  \( |\SYT_{n,2w}| \) is equal to the number of walks of
length $n$ in the region 
  $\{(x_1,\dots,x_w):x_1\ge \dots\ge x_w\ge0\}$ using steps in \( \{ \vec0, \pm
  \epsilon_1,\dots,\pm \epsilon_w\}\) such that zero steps \( \vec0 \) can only occur
  when \( x_w=0 \).

Moreover, this is also true with the refined condition that the
number of odd rows is \( k \) and the number of zero steps is \( k \). 
\end{thm}

While Zeilberger combines a few known results to obtain his theorem,
Eu et al.\ provide bijective proofs of Theorem~\ref{thm:eu} {\it
  and\/} of Theorem~\ref{thm:zeilberger_lazy}.
See however \cite[Appendix~B, proof of Corollary~B.5]{HKKO} for more
conceptual (bijective) proofs using growth diagrams.

Clearly, standard Young tableaux are the special case of semistandard Young
tableaux where the entries are restricted to $1,2,\dots,n$, each
appearing exactly once. Consequently, our results in
Section~\ref{sec:comb-interpr} should imply identities between
standard Young tableaux and the corresponding specializations of the
(marked) up-down tableaux that were the combinatorial objects in that
section. 
Section~\ref{sec:SYT_walks} is devoted to making the corresponding
results explicit, which sometimes requires additional (combinatorial)
arguments. Theorems~\ref{thm:SYTodd} and~\ref{thm:SYTeven} present
our results, the former concerning standard Young tableaux with an
odd bound on the number of columns, the latter with an even bound.

\medskip
{\sc Overview --- Section~\ref{sec:questions}: final questions.}
We close our article by posing a few questions that are suggested
by some of our results.

\section{Definitions}
\label{sec:pre}

In this section we collect definitions concerning
partitions, tableaux and symmetric functions.

For a nonnegative integer \( n \), we write \( [n] = \{ 1,\dots,n\} \).

\medskip
A \emph{partition} of a nonnegative integer \( n \) is a weakly decreasing 
sequence \( \lambda=(\lambda_1, \lambda_2, \dots, \lambda_k) \) of positive 
integers, called \emph{parts}, such that \( \sum_{i\geq 1} \la_i = n \). 
This also includes the empty partition \( () \), denoted by \(
\emptyset \), which, by definition, is the only partition of~0. 
We write \( \lambda\vdash n \) to mean that \( \lambda \) is a partition of \( n \).
If \( \lambda \) is a partition of \( n \) into \( k \) parts, we write 
\( |\lambda|=n \) and \( \ell(\lambda)=k \), and say that \( \lambda \) has 
\emph{size}~\( n \) and \emph{length}~\( k \). 
We denote by \(\Par \) the set of all partitions. 
It is often convenient to identify a partition \( (\lambda_1,\lambda_2,\dots,\lambda_k) \) 
with a sequence \( (\lambda_1,\lambda_2,\dots,\lambda_k, 0, 0, \dots) \)
or \( (\lambda_1,\lambda_2,\dots,\lambda_k, 0, \dots,0) \),
where infinitely or finitely many zeros are appended at the end.
Using this convention, we define \( \lambda_i=0 \) for \( i >\ell(\lambda) \).

The \emph{Young diagram} of a partition \( \lambda=(\la_1,\la_2,\dots, \la_k) \)
is the set \( \{ (i,j) \in \Z^2: 1\le i\le k, 1\le j\le \lambda_i\} \). Each element \(
(i,j) \) in a Young diagram is called a \emph{cell}. The Young diagram of \(
\lambda \) is visualized as a left-justified array of unit square cells with \(
\lambda_i \) cells in the \( i \)-th row, \( i=1,2,\dots, k \), from top to
bottom. We identify \( \lambda \) with its Young diagram. The \emph{conjugate}
(or \emph{transpose}) \( \lambda' \) of a partition \( \lambda \) is the
partition whose Young diagram is given by \( \{(j,i): (i,j)\in \lambda\} \). For
two partitions \( \lambda \) and \( \mu \) we write \( \mu\subseteq\lambda \) to
mean that the Young diagram of \( \mu \) is contained in that of \( \lambda \).
If \( \mu\subseteq\lambda \), the \emph{skew shape} \( \lambda/\mu \) is the
set-theoretic difference \( \lambda-\mu \) of the Young diagrams of \( \lambda
\) and \( \mu \). Each partition \( \lambda \) is also considered as the skew
shape \( \lambda/\emptyset \).

Given a partition $\lambda$, we denote by $r(\lambda)$ (respectively $c(\lambda)$) 
the number of rows (respectively columns) of odd length:
$$
r(\lambda) = | \{ i : \text{$\lambda_i$ is odd} \} |,
\quad
c(\lambda) = | \{ j : \text{$\lambda'_j$ is odd} \} |.
$$

\medskip
A \emph{tableau} of shape \( \lambda/\mu \) is a filling of the cells in \(
\lambda/\mu \) with positive integers. For a tableau \( T \) of shape \(
\lambda=(\lambda_1, \lambda_2, \dots) \), the \emph{size} and
\emph{width} of \( T \) are defined to be \( |\lambda| \) and \(
\lambda_1 \), respectively. A \emph{semistandard Young tableau} (SSYT)
is a tableau in which the entries along rows are weakly increasing and 
the entries along columns are strictly increasing. 
A \emph{standard Young tableau} (SYT) is a semistandard
Young tableau whose entries are the
integers \( 1,2,\dots,n \), where \( n \) is the size of the tableau.

\medskip
A {\it symmetric function} in the variables \( \vx= \{ x_1,x_2,\dots \} \)
is a formal power series in these variables (in the sense that
each monomial appearing in the expansion may only contain a finite number
of variables) that is invariant under
all permutations of the variables that leave almost all variables
invariant. We denote the algebra of symmetric functions by~$\Lambda$.
(Strictly speaking, this is the {\it completion}~$\hat\Lambda$ in the
sense of~\cite[p.~19, Remarks~1]{Macdonald}. However, we do not want
to make this distinction here.)

In this paper we shall be concerned with the following symmetric functions.

For $n\ge1$, the \emph{$n$-th power-sum symmetric function} $p_n(\vx)$ is defined by
$$
p_n(\vx)=\sum_{i\ge1}x_i^n.
$$
The \emph{\( n \)-th complete homogeneous symmetric function} \( h_n(\vx) \) and the \emph{\( n \)-th elementary symmetric function} \( e_n(\vx) \) are defined by
\[
h_n(\vx) =\sum_{i_1\leq i_2 \leq \cdots \leq i_n} x_{i_1}x_{i_2}\cdots x_{i_n}
\quad\text{and}\quad 
e_n(\vx) =\sum_{i_1<i_2<\cdots<i_n} x_{i_1}x_{i_2}\cdots x_{i_n},
\]
respectively. By convention, we set \( h_0(\vx)=e_0(\vx)=1 \) and define \(
h_n(\vx)\) and \(e_n(\vx) \) to be zero for \( n <0 \).

For any tableau \( T\), let \( \vx^T=x_1^{\alpha_1}x_2^{\alpha_2}\cdots
\), where \( \alpha_i \) is the number of \( i \)'s in \( T \). For a partition
\( \lambda \), the \emph{Schur function} \( s_\lambda(\vx) \) is defined by
\begin{equation} \label{eq:SF}
s_{\lambda}(\vx)=\sum_{T} \vx^T,
\end{equation}
where the sum is over all semistandard Young tableaux \( T \) of shape \(
\lambda \).
The number of standard Young tableaux of shape $\lambda$ is equal to 
the coefficient of $x_1 x_2 \cdots x_n$ in the Schur function
$s_\lambda(\vx)$, where $n = |\lambda|$.

The Schur function satisfies determinantal formulas in terms of
complete homogeneous symmetric functions and elementary symmetric functions.
These formulas are known as the {\it Jacobi--Trudi identity}
(cf.\ \cite[p.~41, Eq.~(3.4)]{Macdonald} or \cite[Theorem~7.16.1]{StanBI})
and the {\it N\"agelsbach--Kostka identity}
(cf.\ \cite[p.~41, Eq.~(3.5)]{Macdonald} or
\cite[Corollary~7.16.2]{StanBI}).
Explicitly, they are
\begin{equation}
\label{eq:JT}
s_\lambda(\vx)
 = 
\det_{1 \le i, j \le p} \left( h_{\lambda_i - i + j}(\vx) \right)
 =
\det_{1 \le i, j \le q} \left( e_{\lambda'_i - i + j}(\vx) \right),
\end{equation}
where $p \ge \ell(\lambda)$ and $q \ge \lambda_1$.

The \emph{Hall
  inner product} \( \langle \cdot,\cdot\rangle \) on symmetric
functions is defined by \( \langle
s_\lambda(\vx), s_\mu(\vx) \rangle = \delta_{\lambda,\mu}\) for all partitions \( \lambda
\) and \( \mu \). For a given symmetric function \( f(\vx)\in\Lambda \), the
\emph{skewing operator} \( f^\perp:\Lambda\to \Lambda \) is defined
as the adjoint operator of multiplication by~$f(\vx)$ with respect to the
Hall inner product, that is,
\[
  \langle f^\perp r_1(\vx),r_2(\vx) \rangle =\langle r_1(\vx),f(\vx)r_2(\vx) \rangle,
\]  
for all \( r_1(\vx),r_2(\vx) \in \Lambda \).

\section{Bounded Littlewood identities with a fixed number of odd
  columns:
  proof of Theorem~\ref{thm:row_Goulden}}
\label{sec:rowGoulden}

The purpose of this section is to prove Theorem~\ref{thm:row_Goulden},
which contains refinements of the two bounded Littlewood identities 
in Theorem~\ref{thm:BK-intro} according to the number of odd-length
columns.  

Our proof is based on the minor summation formula, which we recall
in Subsection~\ref{sec:IsWa}.
We need however several further preparations before we are able to
actually prove Theorem~\ref{thm:row_Goulden}.
These are formulas which reduce (certain) Pfaffians
to determinants of matrices of half the size, and which are similar
to formulas that Gordon had given in~\cite{Gordon5} --- see
Lemma~\ref{lem:Gordon} in Subsection~\ref{sec:Gordon} ---, a
particular determinant
simplification --- see Lemma~\ref{lem:sum_det} in the same subsection
---, and formulas for minors of certain special Pfaffians that serve
as the starting point of the use of the minor summation formula ---
see Proposition~\ref{prop:subPf} in Subsection~\ref{sec:subPf}.
With these auxiliary results, Theorem~\ref{thm:row_Goulden} is
finally established in Subsection~\ref{sec:PF}.

\subsection{The minor summation formula}
\label{sec:IsWa}

We use the following notation for submatrices.
Let \( A=(a_{i,j})_{i\in I, j\in J} \) be a matrix 
and let \( R=(r_1,\dots,r_p) \) and \( S=(s_1,\dots,s_q) \) be sequences 
of row and column indices, respectively.
We define
\[
A^R_S = \left( a_{r_i,s_j} \right)_{1 \le i \le p,1 \le j \le q}.
\]
If $A$ is skew-symmetric, then we write $A^R$ instead of~$A^R_R$ for brevity.
We also define
\[ (r_1, \dots, r_p) \sqcup (r'_1, \dots, r'_{p'})
 = (r_1,\dots,r_p,r'_1,\dots,r'_{p'}). \]
By abusing notation, we write \( A^{[m]}_S \) to mean that
$[m]= (1,2,\dots,m)$ is a sequence rather than the set
\( \{ 1,\dots,m\} \).

The minor summation formula of Ishikawa and
Wakayama is the following.

\begin{thm}[{\sc \cite[\sc Theorem~1]{IsWaAA}}]
\label{thm:IsWa}%
Let $m$ be an even integer and $p$ a positive integer (or infinity).
For a $p \times p$ skew-symmetric matrix $A = (a_{r,s})_{1 \le r, s \le p}$ 
and an $m \times p$ matrix $M = ( M_{i,r} )_{1 \le i \le m, 1 \le r \le p}$, we have
\[
\sum_{K}
 \Pf \left( A^K \right)
 \det \left( M_K^{[m]} \right)
 =
\Pf \left( M \,A \, M^t \right)
 =
\Pf_{1 \le i < j \le m}
 \left(
  \sum_{r,s=1}^p a_{r,s} M_{i,r} M_{j,s}
 \right),
\]
where $K = (k_1,\dots,k_m)$ runs over all increasing sequences 
\( 1 \le k_1 < \dots < k_m \le p \) of integers.
\end{thm}

\subsection{Gordon-type reductions of Pfaffians to determinants, and
  a determinant lemma}
\label{sec:Gordon}

The following Gordon-type formulas enable us to transform Pfaffians into determinants.

\begin{lem} \label{lem:Gordon}
If the quantities $z_i$, $i \in \Z$, satisfy $z_{-i} = - z_i$, then we have
\begin{align}
\Pf_{1 \le i, j \le 2w} \left( z_{j-i} \right)
 &=
\det_{1 \le i, j \le w} \left(
 z_{i-j+1} + z_{i-j+3} + \dots + z_{i+j-1}
\right)
\label{eq:G1}\\
 &=
\frac{1}{2}
\det_{1 \le i, j \le w} \left(
 \begin{cases}
  z_i - z_{i-2}, &\text{if $j = 1$} \\
  z_{i-j+1} - z_{i-j-1} + z_{i+j-1} - z_{i+j-3}, &\text{if $2 \le j \le w$}
 \end{cases}
\right)
\label{eq:G2}
\\
 &=
\det_{1 \le i, j \le w} \left(
 \sum_{k=1}^{2j-1} (z_{i-j+k} + z_{i-j+k-1})
\right)
\label{eq:G3}
\\
 &=
\det_{1 \le i, j \le w} \left(
 \sum_{k=1}^{2j-1} (-1)^{k-1} (z_{i-j+k} - z_{i-j+k-1})
\right).
\label{eq:G4}
\end{align}
\end{lem}

\begin{proof}
By \cite[Lemma~1]{Gordon5}, we have
$$
\Pf_{1 \le i, j \le 2w} \Big( z_{j-i} \Big)
 =
\det_{1 \le i, j \le w} \Big( z_{|j-i|+1} + z_{|j-i|+3} + \dots + z_{i+j-1} \Big).
$$
If $i < j$, then we use the relations $z_{-r} + z_r = 0$ to see that
\begin{align*}
z_{i-j+1} + z_{i-j+3} &{}+ \dots + z_{i+j-1}
\\
 &=
z_{i-j+1} + z_{i-j+3} + \dots + z_{j-i-3} + z_{j-i-1}
 +
z_{j-i+1} + z_{j-i+3} + \dots + z_{i+j-3} + z_{i+j-1}
\\
 &=
z_{|j-i|+1} + z_{|j-i|+3} + \dots + z_{i+j-3} + z_{i+j-1}.
\end{align*}
On the other hand, if $i\ge j$ then there is nothing to do.
Hence we obtain~\eqref{eq:G1}.

The identities \eqref{eq:G2}, \eqref{eq:G3} and~\eqref{eq:G4} are deduced 
from~\eqref{eq:G1} by elementary row and column operations.

We first prove \eqref{eq:G2}. By subtracting the $(i-2)$-nd row from
the $i$-th row for $i=w, w-1, \dots, 3$, we obtain
\[
A:= 
\det_{1 \le i, j \le w} \left(
 z_{i-j+1} + z_{i-j+3} + \dots + z_{i+j-1}
\right)
 =
\det_{1 \le i, j \le w} \left(
 \begin{cases}
  z_{i-j+1} + z_{i-j+3} + \dots + z_{i+j-1}, & \text{if \( i\le 2 \)} \\
  z_{i+j-1} - z_{i-j-1}, & \text{if \( i>2 \)}
 \end{cases}
\right).
\]
Note that if \( i=1 \), then 
\[
  z_{i-j+1} + z_{i-j+3} + \dots + z_{i+j-1} =
  z_{-j+2}+z_{-j+4} + \cdots + z_{j-2} + z_{j} = z_j
  = \frac{1}{2} (z_j-z_{-j}),
\]
and if \( i=2 \), then 
\[
  z_{i-j+1} + z_{i-j+3} + \dots + z_{i+j-1} =
  z_{-j+3}+z_{-j+5} + \cdots + z_{j-3} + z_{j-1} +  z_{j+1}
  = z_{j-1} +  z_{j+1} = z_{j+1} - z_{1-j}.
\]
Thus, we have
\[
  A = \frac{1}{2} \det_{1 \le i, j \le w} \left(z_{i+j-1} - z_{i-j-1} \right).
\]
By subtracting the $(j-2)$-nd column from the $j$-th column for
$j=w, w-1, \dots, 3$, we obtain
\[
  A = \frac{1}{2}
\det_{1 \le i, j \le w} \left(
 \begin{cases}
  z_{i+j-1} - z_{i-j-1}, &\text{if \( j\le 2 \)} \\
  z_{i-j+1} - z_{i-j-1} + z_{i+j-1} - z_{i+j-3}, &\text{if \( j>2 \)}
 \end{cases}
\right),
\]
which implies \eqref{eq:G2}.

Equation~\eqref{eq:G3} is derived from~\eqref{eq:G1} 
by adding the $(i-1)$-st row to the $i$-th row for $i=w,w-1, \dots, 2$, 
and then adding the $(j-1)$-st column to the $j$-th column for
$j=w,w-1, \dots, 2$. 
Similarly, Equation~\eqref{eq:G4} is obtained by using subtraction
instead of addition. 
\end{proof}

In the last step of the proof of Theorem~\ref{thm:row_Goulden}, 
we use the following determinant lemma.

\begin{lem}
\label{lem:sum_det}
For $(n+1)$ row vectors $\vv_0, \vv_1, \dots, \vv_n$ of length $n$ 
and scalars $\beta_1, \dots, \beta_n, \gamma_1, \dots, \gamma_n$, we have
$$
\det \begin{pmatrix}
 \beta_1 \vv_0 + \gamma_1 \vv_1 \\
 \beta_2 \vv_1 + \gamma_2 \vv_2 \\
 \vdots \\
 \beta_n \vv_{n-1} + \gamma_n \vv_n
\end{pmatrix}
 =
\sum_{k=0}^n
 \prod_{i=1}^k \beta_i \prod_{j=k+1}^n \gamma_j \cdot
 \det \begin{pmatrix}
  \vv_0 \\ \vdots \\ \vv_{k-1} \\ \vv_{k+1} \\ \vdots \\ \vv_n
 \end{pmatrix}.
$$
\end{lem}

\begin{proof}
For a subset $I$ of $[n]$, 
let $V(I)$ be the $n \times n$ matrix whose $i$-th row is $\vv_{i-1}$ if $i \in I$ 
and $\vv_i$ if $i \not\in I$.
Then we have
$$
\det \begin{pmatrix}
 \beta_1 \vv_0 + \gamma_1 \vv_1 \\
 \beta_2 \vv_1 + \gamma_2 \vv_2 \\
 \vdots \\
 \beta_n \vv_{n-1} + \gamma_n \vv_n
\end{pmatrix}
 =
\sum_{I \subseteq [n]} 
 \prod_{i \in I} \beta_i \prod_{j \not\in I} \gamma_j \cdot
 \det V(I).
$$
If there is an index $p$ such that $p \not\in I$ and $p+1 \in I$, 
then the $p$-th and $(p+1)$-st rows of~$V(I)$ are identical, hence $\det V(I) = 0$.
It follows that $\det V(I) = 0$ unless $I = \emptyset$ or $I = \{ 1,
2, \dots, k \}$ for some $k$  with $1 \le k \le n$, which implies the formula in the lemma.
\end{proof}

\subsection{Sub-Pfaffians and minors}
\label{sec:subPf}
Another key ingredient in the proof of Theorem~\ref{thm:row_Goulden} 
is the following sub-Pfaffian expressions of the weights 
$u^{c(\lambda)} \pm u^{m-c(\lambda)} = u^{r(\lambda')} \pm u^{m-r(\lambda')}$, 
where $\lambda'$ is the conjugate of~$\lambda$, $c(\lambda)$ is the
number of odd columns of~$\lambda$,
and $r(\lambda')$ is the number of odd rows of~$\lambda'$, as before.
For a partition $\mu$ of length $\le m$, we put
$$
I_m(\mu) = (\mu_m+1, \mu_{m-1}+2, \dots, \mu_1+m).
$$

\begin{prop}
\label{prop:subPf}
Let $A = (a_{i,j})$ be the skew-symmetric matrix with rows/columns indexed by 
the totally ordered set $\{ 0 < 0' < 1 < 2 < \cdots \}$, whose entries are given by
$$
a_{0,0'} = 0,
\quad
a_{0,j} = 1+u,
\quad
a_{0',j} = (-1)^{j-1} (1-u),
\quad
a_{i,j} =
\begin{cases}
 1 + u^2, &\text{if $j - i$ is odd,} \\
 2 u, &\text{if $j - i$ is even,}
\end{cases}
$$
where $i$ and $j$ are positive integers and $i<j$.
\begin{enumerate}
\item[(1)]
For a partition $\mu$ of length $\le 2h$, we have
\begin{align*}
\Pf A^{I_{2h}(\mu)}
 &=
 2^{h-1} \left( u^{r(\mu)} + u^{2h-r(\mu)} \right),
\\
\Pf A^{(0,0') \sqcup I_{2h}(\mu)}
 &=
 2^h \left( u^{r(\mu)} - u^{2h-r(\mu)} \right).
\end{align*}
\item[(2)]
For a partition $\mu$ of length $\le 2h+1$, we have
\begin{align*}
\Pf A^{(0) \sqcup I_{2h+1}(\mu)}
 &= 
2^h \left( u^{r(\mu)} + u^{2h+1 - r(\mu)} \right),
\\
\Pf A^{(0') \sqcup I_{2h+1}(\mu)}
 &= 
2^h \left( u^{r(\mu)} - u^{2h+1 - r(\mu)} \right).
\end{align*}
\end{enumerate}
\end{prop}

For the proof of this proposition, we need the following
auxiliary result.

\begin{lem}
\label{lem:PfN}
Let $n$ be an even integer,
and let $N_n = (N_{i,j})_{1 \le i, j \le n}$ be the $n \times n$ skew-symmetric matrix with 
$(i,j)$-entry given by
$$
N_{i,j}
 =
\begin{cases}
 1, &\text{if $j-i$ is odd,} \\
 0, &\text{if $j-i$ is even},
\end{cases}
$$
for $i < j$.
Then we have
$$
\Pf N_n = 2^{n/2-1}.
$$
\end{lem}

\begin{proof}
If \( n=2 \), we have \( \Pf N_n=1 \).
Suppose $n \ge 4$ and proceed by induction on $n$.
By subtracting the third row/column from the first row/column, 
and then by expanding the resulting Pfaffian along the first row/column, 
we see that $\Pf N_n = 2 \Pf N_{n-2}$.
Hence the proof is completed by using the induction hypothesis.
\end{proof}

Proposition~\ref{prop:subPf} is obtained from special cases of the
following lemma. 

\begin{lem}
\label{lem:PfM}
Let $n$ be an even integer.
Let $\vu = (u_1, \dots, u_n)$ be a sequence of indeterminates and 
$\ep = (\ep_1, \dots, \ep_n) \in \{ 1, -1 \}^n$.
We define $M^\ep(\vu) = \big( m_{ij} \big)_{1 \le i, j \le n}$ to be the skew-symmetric matrix 
with $(i,j)$-entry, $i<j$, given by
$$
m_{ij} =
\begin{cases}
 1 + u_i u_j, &\text{if $\ep_i \ep_j = (-1)^{i+j+1}$,} \\
 u_i + u_j,   &\text{if $\ep_i \ep_j = (-1)^{i+j}$.}
\end{cases}
$$
Then we have
$$
\Pf M^\ep(\vu)
 =
2^{n/2-1} \left( 
 \prod_{i:\ep_i = +1} u_i + \prod_{i:\ep_i = -1} u_i
\right).
$$
\end{lem}

\begin{proof}
For a subset $I$ of $[n]$, we compute the coefficient of $u^I = \prod_{i \in I} u_i$ 
in $\Pf M^\ep(\vu)$.
This coefficient is equal to the Pfaffian of the matrix $M[I]$ whose entries are given as follows:

\begin{enumerate}
\item[(a)]
If $\ep_i \ep_j = (-1)^{i+j+1}$, then
$$
M[I]_{i,j} = 
\begin{cases}
 1, &\text{if $(i,j) \in (I \times I) \cup (I^c \times I^c)$,} \\
 0, &\text{if $(i,j) \in (I^c \times I) \cup (I \times I^c)$.}
\end{cases}
$$
\item[(b)]
If $\ep_i \ep_j = (-1)^{i+j}$, then
$$
M[I]_{i,j} = 
\begin{cases}
 0, &\text{if $(i,j) \in (I \times I) \cup (I^c \times I^c)$,} \\
 1, &\text{if $(i,j) \in (I^c \times I) \cup (I \times I^c)$.}
\end{cases}
$$
\end{enumerate}
Here $I^c$ denotes the complement of~$I$ in $[n]$.
Then it suffices to show that
$$
\Pf M[I] = 
\begin{cases}
 2^{n/2-1}, &\text{if $I = E^+$ or $E^-$,} \\
 0, &\text{otherwise,}
\end{cases}
$$
where we put
$$
E^+ = \{ i : \ep_i = 1 \},
\quad
E^- = \{ i : \ep_i = -1 \}.
$$

First we consider the case where $I = E^+$ or $I = E^-$.
In this case, we see that $M[I] = N_n$, thus
we have $\Pf M[I] = 2^{n/2-1}$ by Lemma~\ref{lem:PfN}.

Next we prove that $\Pf M[I] = 0$ unless $I = E^+$ or $E^-$.
Given a subset $K \subseteq [n]$, we put
$$
C(K)
 = 
\left\{ i \in [n-1] : 
 (i,i+1) \in (K \times K) \cup (K^c \times K^c)
\right\}.
$$
Then it is clear that $C(K) = C(L)$ if $K=L^c$.
Conversely, if $C(K) = C(L)$, 
then we can prove, by induction on $k$, that 
$[k] \cap K = [k] \cap L$, for $k=1, 2, \dots, n$, 
or $[k] \cap K = [k] \cap L^c$, for $k=1, 2, \dots, n$.
Hence we have $K = L$ or $K= L^c$.

Suppose $I \neq E^+$ or $E^-$, i.e., $C(I) \neq C(E^+) = C(E^-)$.
Then there is an index $i \in [n-1]$ satisfying one of the following conditions:
\begin{enumerate}
\item[(i)]
$(i,i+1) \in (I \times I) \cup (I^c \times I^c)$ and $\ep_i = - \ep_{i+1}$;
\item[(ii)]
$(i,i+1) \in (I \times I^c) \cup (I^c \times I)$ and $\ep_i = \ep_{i+1}$.
\end{enumerate}
In each of these cases, we can check that the $i$-th row/column of~$M[I]$ 
is identical with the $(i+1)$-st row/column.
Hence we obtain $\Pf M[I] = 0$.
\end{proof}

\begin{proof}[Proof of Proposition~\ref{prop:subPf}]
Given a partition $\mu$ of length $\le m$, we define $(\ep_1, \dots, \ep_m)$ by
$$
\ep_i
 =
(-1)^{\mu_{m+1-i}},
\qquad\text{for }1 \le i \le m.
$$
Then we have $r(\mu) = | \{ i : \ep_i = -1 \} |$ and $m-r(\mu) = | \{ i : \ep_i = 1 \} |$.
It is straightforward to check that
\begin{align*}
A^{I_{2h}(\mu)}
 &= 
M^{(\ep_1, \dots, \ep_{2h})}(u, \dots, u),
\\
A^{(0,0') \sqcup I_{2h}(\mu)}
 &=
M^{(1,1,\ep_1, \dots, \ep_{2h})}(1, -1, u, \dots, u),
\\
A^{(0) \sqcup I_{2h+1}(\mu)}
 &= 
M^{(1, \ep_1, \dots, \ep_{2h+1})}(1, u, \dots, u),
\\
A^{(0') \sqcup I_{2h+1}(\mu)}
 &= 
M^{(1, \ep_1, \dots, \ep_{2h+1})}(-1, u, \dots, u).
\end{align*}
Hence the proof is completed by applying Lemma~\ref{lem:PfM}.
\end{proof}

The next lemma collects several special instances of
the Jacobi--Trudi identity~\eqref{eq:JT}.

\begin{lem}
\label{lem:minor}
Let $T$ be the following matrix with rows indexed by $0, 0', 1, \dots, m$ and columns indexed by 
$0, 0', 1, 2, \dots$:
$$
T = \begin{pmatrix}
 1 & 0 & O \\
 0 & 1 & O \\
 O & O & \left( e_{r-i} \right)_{1 \le i \le m, r \ge 1}
\end{pmatrix},
$$
where each \( O \) is a block matrix consisting of zeros. 
Then the following properties hold.
\begin{enumerate}
\item[(1)]
For an increasing subsequence $K$ of $(1,2, \dots)$ of length $m$, we have
$$
\det T^{[m]}_K = 
\begin{cases}
 s_{\mu'}(\vx) & \text{if $K = I_m(\mu)$ for some partition \( \mu \) of length \( \le m \)} ,\\
 0, & \text{otherwise.}
\end{cases}
$$
\item[(2)]
For an increasing subsequence $K$ of $(0,0',1,2, \dots)$ of length $m+2$, we have
$$
\det T^{(0,0') \sqcup [m]}_K
 =
\begin{cases}
 s_{\mu'}(\vx), &\text{if $K = (0,0') \sqcup I_m(\mu)$ for some partition \( \mu \) of length \( \le m \),} \\
 0, &\text{otherwise.}
\end{cases}
$$
\item[(3)]
For an increasing subsequence $K$ of $(0,1,2, \dots)$ of length $m+1$, we have
$$
\det T^{(0) \sqcup [m]}_K
 =
\begin{cases}
 s_{\mu'}(\vx), &\text{if $K = (0) \sqcup I_m(\mu)$ for some partition \( \mu \) of length \( \le m \),} \\
 0, &\text{otherwise.}
\end{cases}
$$
\item[(4)]
For an increasing subsequence $K$ of $(0',1,2, \dots)$ of length $m+1$, we have
$$
\det T^{(0') \sqcup [m]}_K
 =
\begin{cases}
 s_{\mu'}(\vx), &\text{if $K = (0') \sqcup I_m(\mu)$ for some partition \( \mu \) of length \( \le m \),} \\
 0, &\text{otherwise.}
\end{cases}
$$
\end{enumerate}
\end{lem}

\subsection{Proof of Theorem~\ref{thm:row_Goulden}}
\label{sec:PF}

Now we are in the position to give a proof of Theorem~\ref{thm:row_Goulden}.
Throughout the proof, let $(\alpha_r)_{r \in \Z}$ and $(c_r)_{r \in \Z}$ be the sequences given by
$$
\alpha_r
 =
\begin{cases}
 1+u^2, &\text{if $r$ is odd,} \\
 2u,    &\text{if $r$ is even,}
\end{cases}
\quad r > 0,
\qquad
\alpha_0 = 0,
\quad
\alpha_{-r} = - \alpha_r,
$$
and
$$
c_r = \sum_{k \in \Z} \alpha_k f_{r-k},
$$
where we recall from~\eqref{eq:f_k} that $f_r =f_r(\vx)= \sum_{i \in \ZZ}
e_i(\vx) e_{i+r}(\vx)$, which implies \( f_{-r}=f_r \). 
Note that the $(r,s)$-entry of the skew-symmetric matrix $A$ given in Proposition~\ref{prop:subPf} 
is equal to $\alpha_{s-r}$ for positive integers $r$ and $s$, 
and $c_{-r} = - c_r$.
Moreover, we have
\begin{equation}\label{eq:6}
\sum_{r,s \ge 1} \alpha_{s-r} e_{r-i}(\vx) e_{s-j}(\vx) = c_{j-i}.
\end{equation}

The proof of the identities in Theorem~\ref{thm:row_Goulden} works as follows:

\begin{enumerate}
\item[(a)]
We use Proposition~\ref{prop:subPf} and Lemma~\ref{lem:minor} 
and apply the minor summation formula (Theorem~\ref{thm:IsWa}) 
to express the left-hand side as a Pfaffian.
\item[(b)]
We use one of the Gordon-type identities in Lemma~\ref{lem:Gordon}
to transform the Pfaffian obtained in~(a) into a determinant.
\item[(c)]
We perform row operations with the determinant obtained in (b) 
and use Lemma~\ref{lem:sum_det} to obtain the right-hand side.
\end{enumerate}

\begin{proof}[Proof of Theorem~\ref{thm:row_Goulden}]
First we give a proof of \eqref{eq:even+}.
By applying the minor summation formula to the matrices 
$A$ and $T$ in Proposition~\ref{prop:subPf} and
Lemma~\ref{lem:minor}, where \( m=2w \),
with the $0$-th and $0'$-th rows/columns removed, we have
\begin{align*}
\sum_{\lambda : \lambda_1 \le 2w}
 \left( u^{c(\lambda)} + u^{2w-c(\lambda)} \right) s_{\lambda}(\vx)
 &=
\sum_{\mu : \ell(\mu) \le 2w}
 \left( u^{r(\mu)} + u^{2w-r(\mu)} \right) s_{\mu'}(\vx)
   =
   2^{1-w}
\sum_{\mu : \ell(\mu) \le 2w}
 \Pf A^{I_{2w}(\mu)}
 \det T^{[2w]}_{I_{2w}(\mu)}
\\
 &=
   2^{1-w}
\sum_K \Pf A^K \det T^{[2w]}_K
 =
   2^{1-w}
\Pf_{1 \le i, j \le 2w} \left( \sum_{r,s \ge 1} A_{r,s} T_{i,r} T_{j,s} \right)
   \\
 &=
   2^{1-w}
\Pf_{1 \le i, j \le 2w} \left( c_{j-i} \right),
\end{align*}
where $K$ runs over all increasing subsequences of \( 2w \) positive integers.
Then we use Identity~\eqref{eq:G2} to obtain
$$
\Pf_{1 \le i, j \le 2w} \left( c_{j-i} \right)
 =
\frac{1}{2}
\det_{1 \le i, j \le w} \left(
 \begin{cases}
  c_i - c_{i-2}, &\text{if $j = 1$} \\
  c_{i-j+1} - c_{i-j-1} + c_{i+j-1} - c_{i+j-3}, &\text{if $2 \le j \le w$}
 \end{cases}
\right).
$$
Since $c_r = \sum_{k \ge 1} \alpha_k \left( f_{r-k} - f_{r+k} \right)$ 
with
$\alpha_{2p-1} = 1+u^2$
and $\alpha_{2p} = 2u$ for $p \ge 1$, 
we have
  $$
c_r - c_{r-2}
 =
2 (1+u^2) f_{r-1} + 2 u (f_{r-2} + f_{r}).
$$
Hence, by taking the factor $2$ out of each column except for the first column, 
we obtain
$$
\Pf_{1 \le i, j \le 2w} \left( c_{j-i} \right)
 =
\frac{1}{2}
\cdot 2^{w-1}
\det_{1 \le i, j \le w} \left(
 \begin{cases}
  (1+u^2) (f_{1-j} + f_{j-1}) + 2u (f_{2-j} + f_j),
  &\text{if $i=1$} \\
  (1+u^2) (f_{i-j} + f_{i+j-2}) 
  +
  u (f_{i-j-1} + f_{i+j-3} + f_{i-j+1} + f_{i+j-1}),
  &\text{if $i \ge 2$}
 \end{cases}
\right).
$$
In the last determinant, we subtract the first row multiplied by $u/(1+u^2)$ from the second row, 
and then subtract the $(i-1)$-st row multiplied by $u (1+u^{2i-4})/(1+u^{2i-2})$ 
from the $i$-th row for $i=3, 4, \dots, w$.
Then we see that the last determinant equals
$$
\det_{1 \le i, j \le w} \left(
 \begin{cases}
  (1+u^2) (f_{1-j} + f_{j-1}) + 2u (f_{2-j} + f_j),
   &\text{if $i=1$} \\
  \dfrac{1+u^{2i}}{1+u^{2i-2}} (f_{i-j} + f_{i+j-2})
  +
  u (f_{i-j+1} + f_{i+j-1}),
   &\text{if $i \ge 2$}
 \end{cases}
\right).
$$
Now the proof of \eqref{eq:even+} is completed by applying Lemma~\ref{lem:sum_det} 
with $\beta_1 = 1+u^2$, $\beta_i = (1+u^{2i})/(1+u^{2i-2})$ for $2
\le i \le w$, 
$\gamma_1 = 2u$, $\gamma_i = u$ for $2 \le i \le w$, and 
the row vectors $\vv_i = \left( f_{i-j+1} + f_{i+j-1} \right)_{1 \le j
  \le w}$ for $0 \le i \le w$.

\medskip
Next we prove \eqref{eq:even-}.
By applying the minor-summation formula to the matrices $A$ and $T$ 
in Proposition~\ref{prop:subPf} and Lemma~\ref{lem:minor} with \( m=2w \), we obtain
$$
\sum_{\lambda:\lambda_1 \le 2w}
 \left( u^{c(\lambda)} - u^{2w-c(\lambda)} \right) s_{\lambda}(\vx)
 =
 2^{-w}
\Pf \left( T A T^t \right).
$$
The entries of the skew-symmetric matrix $Q = T A T^t$ are given by
$$
Q_{0,0'} = 0,
\quad
Q_{0,j} = (1+u) e(\vx),
\quad
Q_{0',j} = (-1)^{j-1} (1-u) \overline{e}(\vx),
\quad
Q_{i,j} = c_{j-i},
$$
where \( 1\le i,j\le 2w \), and the last equality follows from
\eqref{eq:6} and the fact that \( T_{k,0} = T_{k,0'} = 0 \) for
\( k \ge 1 \).

By subtracting the $(i-1)$-st row/column from the $i$-th row/column for $i=2w, 2w-1, \dots, 2$, 
and by subsequently expanding the resulting Pfaffian along the $0$-th
row/column, we obtain
$$
\Pf Q = - (1+u) e(\vx) \cdot \Pf Q',
$$
where the entries of $Q'$ are given by
$$
Q'_{0',j} = 2 (-1)^j \overline{e}(\vx),
\quad
Q'_{i,j} = 2 c_{j-i} - c_{j-i+1} - c_{j-i-1}.
$$
By adding the $(i-1)$-st row/column to the $i$-th row/column for $i=2w, 2w-1, \dots, 3$, 
and by subsequently expanding the resulting Pfaffian along the $0'$-th
row/column, we have 
$$
\Pf Q
 =
(1+u) e(\vx) \cdot 2 (1-u) \overline{e}(\vx)
 \cdot \Pf_{3 \le i, j \le 2w} \left( 2 c_{j-i} - c_{j-i+2} - c_{j-i-2} \right).
$$
By using the Gordon-type formula \eqref{eq:G1}, we obtain
\begin{align*}
\Pf_{3 \le i, j \le 2w} \left( 2 c_{j-i} - c_{j-i+2} - c_{j-i-2} \right)
 &=
\det_{1 \le i, j \le w-1} \left( c_{i-j+1} - c_{i-j-1} - c_{i+j+1} + c_{i+j-1} \right)
\\
 &=
2^{w-1} 
\det_{1 \le i, j \le w-1} \left(
 (1+u^2) (f_{i-j} - f_{i+j}) + u (f_{i-j-1} + f_{i-j+1} - f_{i+j-1} - f_{i+j+1} )
\right).
\end{align*}
By subtracting the $(i-1)$-st row multiplied by $u (1-u^{2i-2})/(1-u^{2i})$ from the $i$-th row 
for $i=2, 3, \dots, w-1$, we see that the last determinant equals
$$
\det_{1 \le i, j \le w-1}
\left(
 \frac{1-u^{2i+2}}{1-u^{2i}} (f_{i-j} - f_{i+j}) + u (f_{i-j+1} - f_{i+j+1})
\right).
$$
Now we apply Lemma~\ref{lem:sum_det} with 
$\beta_i = (1-u^{2i+2})/(1-u^{2i})$, $\gamma_i = u$, for $1 \le i \le w-1$,
and $\vv_i = \left( f_{i-j+1} - f_{i+j+1} \right)_{1 \le j \le w-1}$,
for $0 \le i \le w-1$, to obtain~\eqref{eq:even-}.

\medskip
Now we turn to \eqref{eq:odd+}.
We apply the minor summation formula to the matrices $A$ and $T$ 
in Proposition~\ref{prop:subPf} and
Lemma~\ref{lem:minor}, where \( m=2w+1 \),
with the $0'$-th row/column removed to obtain
$$
\sum_{\lambda:\lambda_1 \le 2w+1}
 \left( u^{c(\lambda)} + u^{2w+1-c(\lambda)} \right) s_{\lambda}(\vx)
 =
 2^{-w}
\Pf \begin{pmatrix}
 0 & \left( (1+u) e(\vx) \right)_{1 \le j \le 2w+1} \\
 * & \left( c_{j-i} \right)_{1 \le i, j \le 2w+1}
\end{pmatrix}.
$$
By subtracting the $(i-1)$-st row/column from the $i$-th row/column 
for $i=2w+1, 2w, \dots, 2$, 
and by subsequently expanding the resulting Pfaffian along the $0$-th
row/column, we see that
$$
\Pf \begin{pmatrix}
 0 & \left( (1+u) e(\vx) \right)_{2 \le j \le 2w+1} \\
 * & \left( c_{j-i} \right)_{2 \le i, j \le 2w+1}
\end{pmatrix}
 =
(1+u) e(\vx)
 \cdot
\Pf_{2 \le i, j \le 2w+1} \left(
 2 c_{j-i} - c_{j-i+1} - c_{j-i-1}
\right).
$$
By using \eqref{eq:G3}, we have
\begin{multline*}
\Pf_{2 \le i, j \le 2w+1} \left(
 2 c_{j-i} - c_{j-i+1} - c_{j-i-1}
\right)
\\
 =
2^w \det_{1 \le i, j \le w} \left(
 (1-u+u^2) (f_{i-j} - f_{i+j-1})
  + u (f_{i-j-1} + f_{i-j} + f_{i-j+1} - f_{i+j-2} - f_{i+j-1} - f_{i+j})
\right).
\end{multline*}
Finally we subtract the $(i-1)$-st row multiplied by $u (1+u^{2i-3})/(1+u^{2i-1})$ 
from the $i$-th row for $i=w, w-1, \dots, 2$ to obtain
\begin{multline*}
\det_{1 \le i, j \le w} \left(
 (1-u+u^2) (f_{i-j} - f_{i+j-1})
  + u (f_{i-j-1} + f_{i-j} + f_{i-j+1} - f_{i+j-2} - f_{i+j-1} - f_{i+j})
\right)
\\
 =
\det_{1 \le i, j \le w} \left(
 \frac{1+u^{2i+1}}{1+u^{2i-1}} (f_{i-j} - f_{i+j-1}) + u (f_{i-j+1} - f_{i+j})
\right).
\end{multline*}
The proof of \eqref{eq:odd+} is completed by applying Lemma~\ref{lem:sum_det}.

\medskip
The proof of \eqref{eq:odd-} is the same as~\eqref{eq:odd+} up to a sign, 
so we omit it.
\end{proof}

\section{Equivalence of Theorem~\ref{thm:row_Goulden} and identities
  for classical group characters of nearly rectangular shape}
\label{sec:nearlyrect}

In this section, we explain that Theorem~\ref{thm:row_Goulden} is equivalent 
to the nearly rectangular character identities for the even orthogonal Lie algebra 
$\mathfrak{so}_{2n}$ (Theorem~\ref{thm:Kratt} below) obtained by
the third author in~\cite{Krattenthaler1998}.

\medskip
The finite dimensional irreducible representations of $\mathfrak{so}_{2n}$ are parameterized 
by their highest weights $\lambda$, where $\lambda = (\lambda_1, \dots, \lambda_n)$ 
is a sequence of integers or of half-integers such that
\begin{equation}
\label{eq:so_hw}
\lambda_1 \ge \lambda_2 \ge \dots \ge \lambda_{n-1} \ge |\lambda_n|.
\end{equation}
Let $\vx_n = (x_1, \dots, x_n)$.
We denote by $\sorth_\lambda(\vx_n)$ the character of the irreducible representation 
with highest weight $\lambda$.

The third author used tableau descriptions of the special orthogonal
characters that are derived from Lakshmibai--Seshadri paths
to prove the following character identities.

\begin{thm}
[{\sc \cite[\sc Theorem~2, Eq.~(3.7)]{Krattenthaler1998}}]
\label{thm:Kratt}
If $b$ is a nonnegative integer or half-integer and $0 \le k \le 2b$, then we have
\begin{equation}
\label{eq:Kratt}
\sorth_{(b^{n-1},b-k)}(\vx_n)
 =
\left( x_1 \cdots x_n \right)^{-b}
\underset{c(((2b)^n)/\lambda) = k}{\sum_{\lambda \subseteq ((2b)^n)}}
 s_\lambda(\vx_n).
\end{equation}
In words:
$\lambda$ runs over all partitions contained in the \( n\times (2b) \) rectangle $((2b)^n)$ 
such that the skew diagram $((2b)^n)/\lambda$ has exactly $k$ columns of
odd length. 
\end{thm}

For a sequence $\lambda$ satisfying \eqref{eq:so_hw}, we put
$$
\lambda^{\#} = (\lambda_1, \dots, \lambda_{n-1}, - \lambda_n),
$$
and define
\begin{align*}
\orth_\lambda(\vx_n)
 &=
\begin{cases}
 \sorth_\lambda(\vx_n) + \sorth_{\lambda^{\#}}(\vx_n), &\text{if $\lambda_n \neq 0$}, \\
 \sorth_\lambda(\vx_n), &\text{if $\lambda_n = 0$},
\end{cases}
\\
\overline{\orth}_\lambda(\vx_n)
 &=
\begin{cases}
 \sorth_\lambda(\vx_n) - \sorth_{\lambda^{\#}}(\vx_n), &\text{if $\lambda_n \neq 0$}, \\
 0, &\text{if $\lambda_n = 0$}.
\end{cases}
\end{align*}
Then these Laurent polynomials can be described in terms of 
elementary symmetric polynomials $e_r(\vx_n^{\pm 1}) = e_r(x_1, \dots, x_n, x_1^{-1}, \dots, x_n^{-1})$.
Given a partition $\lambda$ of length $\le n$, 
we write $\lambda+1/2 = (\lambda_1 + 1/2, \dots, \lambda_n + 1/2)$.

\begin{prop}[Cf.~\cite{Koike1997,ProcAK}]
\label{prop:so_JT}
\leavevmode

\begin{enumerate}
\item[(1)]
For a partition $\lambda$ of length $\le n$, we have
$$
\orth_\lambda(\vx_n)
 =
\frac{1}{2}
\det_{1 \le i, j \le \lambda_1} \left(
 e_{\lambda'_i - i + j}(\vx_n^{\pm 1}) + e_{\lambda'_i - i - j +2}(\vx_n^{\pm 1})
\right),
$$
and, if $\ell(\lambda) = n$, we have 
$$
\overline{\orth}_\lambda(\vx_n)
 =
\prod_{i=1}^n \left( x_i - x_i^{-1} \right) \cdot
\det_{2 \le i, j \le \lambda_1} \left(
 e_{\lambda'_i - i + j}(\vx_n^{\pm 1}) - e_{\lambda'_i - i - j}(\vx_n^{\pm 1})
\right).
$$
\item[(2)]
For a partition $\lambda$ of length $\le n$, we have
\begin{align*}
\orth_{\lambda+1/2}(\vx_n)
 &=
\prod_{i=1}^n \left( x_i^{1/2} + x_i^{-1/2} \right) \cdot
\det_{1 \le i, j \le \lambda_1} \left(
 e_{\lambda'_i - i + j}(\vx_n^{\pm 1}) - e_{\lambda'_i - i - j + 1}(\vx_n^{\pm 1})
\right),
\\
\overline{\orth}_{\lambda+1/2}(\vx_n)
 &=
\prod_{i=1}^n \left( x_i^{1/2} - x_i^{-1/2} \right) \cdot
\det_{1 \le i, j \le \lambda_1} \left(
 e_{\lambda'_i - i + j}(\vx_n^{\pm 1}) + e_{\lambda'_i - i - j + 1}(\vx_n^{\pm 1})
\right).
\end{align*}
\end{enumerate}
\end{prop}

Now we prove the equivalence between Theorems~\ref{thm:row_Goulden} and \ref{thm:Kratt}.
Note that a symmetric function identity $f(\vx) = g(\vx)$ holds for infinitely many variables $\vx$ 
if and only if $f(\vx_n) = g(\vx_n)$ for all positive integers~$n$.

First we show that Equations~\eqref{eq:even+} and~\eqref{eq:even-} 
are equivalent to Theorem~\ref{thm:Kratt} in the case where $b$ is a
positive integer. 
Under the specialization $x_{n+1} = x_{n+2} = \dots = 0$, we have
$$
f_r(\vx_n)
 =
\sum_i e_i(\vx_n) \cdot  e_{r+i}(\vx_n)
 =
\sum_i e_i(\vx_n) \cdot (x_1 \cdots x_n) e_{n-r-i}(\vx_n^{-1})
 =
(x_1 \cdots x_n) e_{n-r}(\vx_n^{\pm 1}),
$$
where $\vx_n^{-1} = (x_1^{-1}, \dots, x_n^{-1})$, and
$$
e(\vx_n)
 = 
(x_1 \cdots x_n)^{1/2} \prod_{i=1}^n \left( x_i^{1/2} + x_i^{-1/2} \right),
\quad
\overline{e}(\vx_n)
 = 
(-1)^n (x_1 \cdots x_n)^{1/2} \prod_{i=1}^n \left( x_i^{1/2} - x_i^{-1/2} \right).
$$
Comparing the determinants on the right-hand sides of~\eqref{eq:even+} and~\eqref{eq:even-} 
with the Jacobi--Trudi-type identities in Proposition~\ref{prop:so_JT}(1) with \( \lambda=(w^{n-1},w-k) \), we obtain
$$
\frac{1}{2}
 \det_{1 \le i, j \le w} \left(
  \begin{cases}
   f_{i-j}(\vx_n) + f_{i+j-2}(\vx_n), &\mbox{if \(1 \le i \le w-k \)} \\
   f_{i-j+1}(\vx_n) + f_{i+j-1}(\vx_n), &\mbox{if \( w-k+1 \le i \le w \)}
  \end{cases}
 \right)
 =
(x_1 \cdots x_n)^w \cdot \orth_{(w^{n-1}, w-k)}(\vx_n)
$$
and
$$
e(\vx_n) \overline{e}(\vx_n) \cdot 
 \det_{1 \le i, j \le w-1} \left(
  \begin{cases}
   f_{i-j}(\vx_n) - f_{i+j}(\vx_n), &\mbox{if \( 1 \le i \le w-k-1 \)} \\
   f_{i-j+1}(\vx_n) - f_{i+j+1}(\vx_n), &\mbox{if \( w-k \le i \le w-1 \)}
  \end{cases}
 \right)
 =
(-1)^n (x_1 \cdots x_n)^w \cdot
\overline{\orth}_{(w^{n-1},w-k)}(\vx_n).
$$
By the definition of $\orth_\lambda(\vx_n)$ and $\overline{\orth}_\lambda(\vx_n)$, 
we have
\begin{align*}
\sum_{k=0}^w \left( u^k + u^{2w-k} \right) \orth_{(w^{n-1}, w-k)}(\vx_n)
 &=
\sum_{k=0}^{2w} \left( u^k + u^{2w-k} \right) \sorth_{(w^{n-1}, w-k)}(\vx_n),
\\
\sum_{k=0}^{w-1} \left( u^k - u^{2w-k} \right) \overline{\orth}_{(w^{n-1}, w-k)}(\vx_n)
 &=
\sum_{k=0}^{2w} \left( u^k - u^{2w-k} \right) \sorth_{(w^{n-1}, w-k)}(\vx_n).
\end{align*}
Hence \eqref{eq:even+} and \eqref{eq:even-} become
\begin{align*}
(x_1 \cdots x_n)^{-w}
\sum_{\lambda_1 \le 2w} (u^{c(\lambda)} + u^{2w-c(\lambda)}) s_{\lambda}(\vx_n)
 &=
\sum_{k=0}^{2w} \left( u^k + u^{2w-k} \right) \sorth_{(w^{n-1}, w-k)}(\vx_n),
\\
(-1)^n (x_1 \cdots x_n)^{-w}
\sum_{\lambda_1 \le 2w} (u^{c(\lambda)} - u^{2w-c(\lambda)}) s_{\lambda}(\vx_n)
 &=
\sum_{k=0}^{2w} \left( u^k - u^{2w-k} \right) \sorth_{(w^{n-1}, w-k)}(\vx_n).
\end{align*}
This pair of identities is equivalent to
$$
\sum_{k=0}^{2w} u^k \sorth_{(w^{n-1}, w-k)}(\vx_n)
 =
(x_1 \cdots x_n)^{-w}
\begin{cases}
 \sum_{\lambda_1 \le 2w} u^{c(\lambda)} s_{\lambda}(\vx_n), &\text{if $n$ is even,} \\
 \sum_{\lambda_1 \le 2w} u^{2w-c(\lambda)} s_{\lambda}(\vx_n), &\text{if $n$ is odd.}
\end{cases}
$$
If $\lambda \subseteq ((2w)^n)$, 
then
$$
c((2w)^n / \lambda)
 = 
\begin{cases}
 c(\lambda), &\text{if $n$ is even,} \\
 2w - c(\lambda), &\text{if $n$ is odd.}
\end{cases}
$$
Therefore we obtain
$$
\sum_{k=0}^{2w} u^k \sorth_{(w^{n-1}, w-k)}(\vx_n)
 =
(x_1 \cdots x_n)^{-w}
\sum_{\lambda \subseteq ((2w)^n)} u^{c((2w)^n/\lambda)} s_\lambda(\vx_n),
$$
which is the generating function expression of the third author's
formula~\eqref{eq:Kratt}  
in the case where $b=w$ is a positive integer.

\medskip
By using the Jacobi--Trudi-type identities in Proposition~\ref{prop:so_JT}(2),
we can show that the pair of identities~\eqref{eq:odd+} and~\eqref{eq:odd-} 
is equivalent to the third author's formula~\eqref{eq:Kratt} 
in the case where $b$ is a positive half-integer.

\begin{rem}
Along a similar argument, we can see that 
Goulden's original identity~\eqref{eq:Goulden2m,k} is equivalent to 
the nearly rectangular character identity for the symplectic Lie algebra 
\cite[Theorem~2, Eq.~(3.6)]{Krattenthaler1998}.
\end{rem}

\section{Bounded Littlewood identities and skewing operators:
  proofs of Theorems~\ref{thm:BK2},
\ref{thm:Goulden2}, and~\ref{thm:row_Goulden2-2pop}}\label{sec:skew}

In this section, we provide the proofs of the
formulations of the (refined) bounded Littlewood identities
in terms of skewing operators, given in Theorems~\ref{thm:BK2},
\ref{thm:Goulden2}, and~\ref{thm:row_Goulden2-2pop}.

\medskip
We will use the following restatement of \Cref{thm:row_Goulden}.

\begin{thm}\label{thm:row_Goulden2-2}
For  nonnegative integers \( 0\le k\le w \),
we have
\begin{align}
  \label{eq:schur-odd-k}
  \underset {c(\lambda) = k}
  {\sum_{\lambda:\lambda_1 \le 2w+1}} s_{\lambda}(\vx)
  +\underset {c(\lambda) = 2w+1-k}
  {\sum_{\lambda:\lambda_1 \le 2w+1}} s_{\lambda}(\vx)
  &= e(\vx) 
 \det_{1 \le i, j \le w} \left(
   f_{i+\chi(i>w-k)-j}(\vx) - f_{i+\chi(i>w-k)+j-1}(\vx)
   \right),\\
  \label{eq:schur-odd-2h+1-k}
  \underset {c(\lambda) = k}
  {\sum_{\lambda:\lambda_1 \le 2w+1}} s_{\lambda}(\vx)
  -\underset {c(\lambda) = 2w+1-k}
  {\sum_{\lambda:\lambda_1 \le 2w+1}} s_{\lambda}(\vx)
  &= \overline{e}(\vx) 
 \det_{1 \le i, j \le w} \left(
   f_{i+\chi(i>w-k)-j}(\vx) + f_{i+\chi(i>w-k)+j-1}(\vx)
   \right),\\
  \label{eq:schur-even-k}
  \underset {c(\lambda) = k}
  {\sum_{\lambda:\lambda_1 \le 2w}} s_{\lambda}(\vx)
  +\underset {c(\lambda) = 2w-k}
  {\sum_{\lambda:\lambda_1 \le 2w}} s_{\lambda}(\vx)
  &= \left(\frac{1}{2} \right)^{\chi(k<w) }
 \det_{1 \le i, j \le w} \left(
   f_{i+\chi(i>w-k)-j}(\vx) + f_{i+\chi(i>w-k)+j-2}(\vx)
   \right),
\end{align}
and, for  \( 0\le k\le w-1 \),
\begin{align}
  \label{eq:schur-even-2h-k}
  \underset {c(\lambda) = k}
  {\sum_{\lambda:\lambda_1 \le 2w}} s_{\lambda}(\vx)
  -\underset {c(\lambda) = 2w-k}
  {\sum_{\lambda:\lambda_1 \le 2w}} s_{\lambda}(\vx)
  &= e(\vx) \overline{e}(\vx) 
 \det_{2 \le i, j \le w} \left(
   f_{i+\chi(i>w-k)-j}(\vx) - f_{i+\chi(i>w-k)+j-2}(\vx)
   \right).
\end{align}
\end{thm}

We begin with an auxiliary result that describes the action of the
skewing operator $p_1^\perp$ on the series 
$f_i(\vx)$ given in \eqref{eq:f_k}. We remind the reader that the
definition of the skewing operators is given at the end of
Section~\ref{sec:pre}.
  
\begin{lem} \label{lem:pop} 
For integers \( i \) and \( j\geq 0 \), we have
\begin{equation}\label{eq:popLemma}
  (\pop)^{j} f_i(\vx) = \sum_{r=0}^{j} \binom{j}{r}f_{i-j+2r}(\vx).
\end{equation}
In particular, \( \pop f_i(\vx)=f_{i-1}(\vx)+f_{i+1}(\vx) \).
\end{lem}

\begin{proof}
 For the case \( j =0 \) the claim is obvious. Now we use induction on \( j\geq 1 \). 
 For \( j=1 \), since \( \pop \) is a derivation, we have
  \begin{align*}
  \pop f_i(\vx)
    &= \sum_{n\in\ZZ} \left ((\pop e_n(\vx)) e_{n+i}(\vx) + e_n(\vx) (\pop  e_{n+i}(\vx)) \right)\\
    &= \sum_{n\in\ZZ} \left( e_{n-1}(\vx) e_{n+i}(\vx) + e_n(\vx) e_{n+i-1}(\vx) \right)\\
    &= f_{i+1}(\vx)+f_{i-1}(\vx).
  \end{align*} 
To get the second line, we have used that
\( p_k^\perp e_n(\vx) = (-1)^{k-1} e_{n-k}(\vx) \),
as is well known~\cite[p.~76]{Macdonald}.

  Suppose that \eqref{eq:popLemma} holds for some \( j\geq 1 \). Then,
  \begin{align*}
    (\pop)^{j+1} f_i(\vx)
    &= (\pop)^{j}(f_{i+1}(\vx)+f_{i-1}(\vx))\\
    &= \sum_{r=0}^{j}  \binom{j}{r}f_{(i+1)-j+2r}(\vx)
      + \sum_{r=0}^{j}\binom{j}{r}f_{(i-1)-j+2r}(\vx)\\
    &= \sum_{r=1}^{j+1}  \binom{j}{r-1}f_{i-(j+1)+2r}(\vx)
      + \sum_{r=0}^{j}\binom{j}{r}f_{i-(j+1)+2r}(\vx)\\
     &= \sum_{r=0}^{j+1} \binom{j+1}{r}f_{i-(j+1)+2r}(\vx)
  \end{align*}
 so that \eqref{eq:popLemma} holds for all \( j\geq 0\). 
 \end{proof}

This lemma has the following consequence.

\begin{lem}\label{lem:col_op}
  For any integers \( i \) and \( j\geq 1 \),  
  \begin{align}
  \label{eq:pop_h}
     (\pop)^{j-1} (f_{i-1}(\vx) \pm f_{i+1}(\vx))
     &= f_{i-j}(\vx)\pm f_{i+j}(\vx)
     + \sum_{r=1}^{j-1} \binom{j-1}{r}
     (f_{i-(j-2r)}(\vx) \pm f_{i+(j-2r)}(\vx)),\\
  \label{eq:pop_p}
    (\pop)^{j-1}(f_{i-1}(\vx)\pm f_{i}(\vx))
     &=f_{i-j}(\vx)\pm f_{i+j-1}(\vx) 
     +\sum_{r=1}^{j-1}\binom{j-1}{r} (f_{i-(j-2r)}(\vx)\pm f_{i+(j-2r)-1}(\vx)).
  \end{align}
\end{lem}

\begin{proof}
  By Lemma \ref{lem:pop}, the left-hand side of~\eqref{eq:pop_h} is equal to
  \[
    \sum_{r=0}^{j-1}  \binom{j-1}{r}f_{(i-1)-(j-1)+2r}(\vx)
    \pm \sum_{r=0}^{j-1} \binom{j-1}{r}f_{(i+1)-(j-1)+2r}(\vx) 
    =\sum_{r=0}^{j-1}  \binom{j-1}{r}f_{i-j+2r}(\vx)
    \pm \sum_{r=0}^{j-1} \binom{j-1}{j-1-r}f_{i+j-2r}(\vx),
  \]
which is equal to the right-hand side of \eqref{eq:pop_h}.
The second identity~\eqref{eq:pop_p} can be proved similarly.
\end{proof}

With Lemma~\ref{lem:col_op}, we are now in the position to
prove Theorems~\ref{thm:BK2}--\ref{thm:row_Goulden2-2pop}.

\begin{proof}[Proof of Theorem~\ref{thm:BK2}] 
  By \eqref{eq:pop_h} we have
  \[
    (\pop)^{j-1} (f_{i-1}(\vx)-f_{i+1}(\vx))
    = f_{i-j}(\vx)-f_{i+j}(\vx)
    +\sum_{r=1}^{j-1} \binom{j-1}{r}
    (f_{i-(j-2r)}(\vx)-f_{i+(j-2r)}(\vx)).
  \]
  Note that, if \( 1\le r \le j-1 \), then \( f_{i-(j-2r)}(\vx)-f_{i+(j-2r)}(\vx) \) is equal to either 
  \( 0 \) or \(\pm(f_{i-j'}(\vx)-f_{i+j'}(\vx)) \) for some $j'$
  with \( 1\le j'\le j-1 \). 
  Hence the matrix \( ( (\pop)^{j-1} (f_{i-1}(\vx)-f_{i+1}(\vx)) )_{1\le i,j \le w} \) can be obtained from the matrix 
  \( ( f_{i-j}(\vx)-f_{i+j}(\vx) )_{1\le i,j \le w} \) by adding a multiple of column \( j' \) to column \( j \) for \( 1\le j'\le j-1 \). 
By~\eqref{eq:BK_odd1}, this implies that
  \begin{equation}
    \label{eq:oddref1}
   \sum_{\la:\lambda_1\le 2w+1} s_{\lambda}(\vx)
    =e(\vx)  \det_{1\le i,j \le w} \left( f_{i-j}(\vx)-f_{i+j}(\vx) \right)
    =e(\vx)  \det_{1\le i,j\le w} \left( (\pop)^{j-1} (f_{i-1}(\vx)-f_{i+1}(\vx)) \right).
  \end{equation}
Again, by \eqref{eq:pop_h} we also obtain
\[
  (\pop)^{i-1} (f_{0}(\vx)-f_{2}(\vx))
  = f_{1-i}(\vx)-f_{1+i}(\vx)
  +\sum_{r=1}^{i-1} \binom{i-1}{r}
  (f_{1-(i-2r)}(\vx)-f_{1+(i-2r)}(\vx)).
\]
Since the operator \( (\pop)^{j-1} \) is linear, by the same argument with row
operations we have
\begin{equation}
  \label{eq:oddref2}
  \det_{1\le i,j\le w} \left( (\pop)^{j-1} (f_{1-i}(\vx)-f_{1+i}(\vx)) \right)
  =\det_{1\le i,j\le w} \left( (\pop)^{j-1}(\pop)^{i-1} (f_{0}(\vx)-f_{2}(\vx)) \right).
\end{equation}
Since \( f_{1-i}(\vx) = f_{i-1}(\vx) \), we obtain the first
identity~\eqref{eq:oddpop} from~\eqref{eq:oddref1} and~\eqref{eq:oddref2}.

\medskip
Now we prove the second identity,
Equation~\eqref{eq:evenpop}. By~\eqref{eq:pop_p} we have 
\[
  (\pop)^{j-1}(f_{i-1}(\vx)+f_{i}(\vx))
  =f_{i-j}(\vx)+f_{i+j-1}(\vx)+\sum_{r=1}^{j-1}\binom{j-1}{r} (f_{i-(j-2r)}(\vx)+f_{i+(j-2r)-1}(\vx)).
\]
For $r$ with \( 1\le r \le j-1 \), let \( j' \coloneqq j-2r \) if
\( 1\le r<\flr{j/2} \),  
and \( j' \coloneqq 2r-j+1 \) if \( \flr{j/2}\le r \le j-1 \).
Then \( f_{i-(j-2r)}(\vx)+f_{i+(j-2r)-1}(\vx)=f_{i-j'}(\vx)+f_{i+j'-1}(\vx) \) and \( 1\le j' \le j-1 \). 
By~\eqref{eq:BK_even1}, this implies that
  \begin{equation}
    \label{eq:p1j}
    \sum_{\lambda:\lambda_1\le 2w} s_{\lambda}(\vx)
    =\det_{1\leq i,j \leq w}\left(f_{i-j}(\vx)+f_{i+j-1}(\vx)\right)=
    \det_{1\leq i,j \leq w}\left( (\pop)^{j-1}(f_{i-1}(\vx)+f_{i}(\vx)) \right).
  \end{equation}
Again, by \eqref{eq:pop_p} we obtain
 \[
   (\pop)^{i-1}(f_{-1}(\vx)+f_{0}(\vx))
   =f_{-i}(\vx)+f_{i-1}(\vx)+\sum_{r=1}^{i-1}\binom{i-1}{r} (f_{-(i-2r)}(\vx)+f_{(i-2r)-1}(\vx)).
 \]
 By the same argument one can show that
 \begin{equation}
   \label{eq:p1i}
   \det_{1\leq i,j \leq w}\left((\pop)^{j-1}(f_{-i}(\vx)+f_{i-1}(\vx))\right)=
   \det_{1\leq i,j \leq w}\left( (\pop)^{j-1}(\pop)^{i-1}(f_{-1}(\vx)+f_{0}(\vx)) \right).
 \end{equation}
Since \( f_{-i}(\vx)=f_{i}(\vx) \), Identity~\eqref{eq:evenpop}
follows from~\eqref{eq:p1j} and~\eqref{eq:p1i}.
\end{proof}

\begin{proof}[Proof of Theorem~\ref{thm:Goulden2}]
  The first identity, Equation~\eqref{eq:oddref_k}, can be proved by the
  same argument as 
  in the proof of~\eqref{eq:oddpop}. 

For the second identity, observe that
\eqref{eq:pop_h} gives
\[
  (\pop)^{j-1} (f_{i+k\delta_{i,w}-1}(\vx)-f_{i+k\delta_{i,w}+1}(\vx))
  = f_{i+k\delta_{i,w}-j}(\vx)-f_{i+k\delta_{i,w}+j}(\vx)
  +\sum_{r=1}^{j-1} \binom{j-1}{r}
  (f_{i+k\delta_{i,w}-(j-2r)}(\vx)-f_{i+k\delta_{i,w}+(j-2r)}(\vx)).
\]
Therefore, we can apply column operations in \eqref{eq:Goulden2m,k} to obtain
 \[
 \underset {r(\lambda)=k}
 {\sum_{\lambda:\lambda_1\le 2w}}s_{\lambda}(\vx)
   =\det_{1\le i,j\le w} 
   \left( f_{i+k\delta_{i,w}-j}(\vx)-f_{i+k\delta_{i,w}+j}(\vx) \right)
   =\det_{1\le i,j\le w} 
   \left( (\pop)^{j-1} (f_{i+k\delta_{i,w}-1}(\vx)-f_{i+k\delta_{i,w}+1}(\vx)) \right),
   \]
which concludes the proof.
\end{proof}

\begin{proof}[Proof of Theorem~\ref{thm:row_Goulden2-2pop}]
We apply column operations in Theorem~\ref{thm:row_Goulden2-2} as in the proof of Theorem~\ref{thm:Goulden2}.
We only give a proof of~\eqref{eq:schur-even-kpop}.
By~\eqref{eq:schur-even-k} and Lemmas~\ref{lem:pop} and
\ref{lem:col_op}, we obtain 
\begin{align*}
  \underset {c(\lambda)=k}
  {\sum_{\lambda:\lambda_1\le 2w}}s_{\lambda}(\vx)
  +\underset {c(\lambda)=2w-k}
  {\sum_{\lambda:\lambda_1\le 2w}}s_{\lambda}(\vx)
   &= \left(\frac{1}{2} \right)^{\chi(k<w) }
 \det_{1 \le i, j \le w} \left(
   f_{i+\chi(i>w-k)-j}(\vx) + f_{i+\chi(i>w-k)+j-2}(\vx)
   \right)\\
   &\kern-8pt
   =2^{\chi(k=w)} \det_{1\le i,j \leq w}\left( \begin{cases}
   f_{i+\chi(i>w-k)-1}(\vx), & \mbox{if \( j=1 \)}\\
   f_{(i-1)+\chi(i>w-k)-(j-1)}(\vx)+f_{(i-1)+\chi(i>w-k)+(j-1)}(\vx), & \mbox{if \( 2\le j \le w \)}
   \end{cases}\right)\\
   &\kern-8pt
   =2^{\chi(k=w)}\det_{1\le i,j \leq w}\left( \begin{cases}
  f_{i+\chi(i>w-k)-1}(\vx), & \mbox{if \( j=1 \)}\\
  (\pop)^{j-2} (f_{i+\chi(i>w-k)-2}(\vx)-f_{i+\chi(i>w-k)}(\vx)),& \mbox{if \( 2\le j \le w \)}
   \end{cases}\right)\\
   &\kern-8pt
   =2^{\chi(k=w)}\det_{1\le i,j \leq w}\left((\pop)^{j-1}f_{i+\chi(i>w-k)-1}(\vx) \right),
\end{align*}
completing the proof.
\end{proof}

\section{Formulas for the number of standard Young tableaux of
  bounded width: proofs of Theorems~\ref{thm:KLO}
  and~\ref{cor:ref_KLO}} \label{sec:SYT-bounded} 

Here we prove the formulas for the number of standard Young tableaux
of bounded width in Theorems~\ref{thm:KLO} and~\ref{cor:ref_KLO},
the former being an alternative to Theorem~\ref{thm:Gessel}, the
latter being a refinement.

\medskip
The important tool here is
the ring homomorphism \( \theta : \Lambda \rightarrow \QQ[[x]] \), defined by
\( \theta(p_1(\vx))=x \) and \( \theta(p_n(\vx))=0 \) for \( n>1 \).
Gessel \cite[Theorem~1]{Gessel1990} 
showed that, for \( f(\vx)\in\Lambda \),
  \[
  \theta(f(\vx))=\sum_{n=0}^{\infty}a_n\frac{x^n}{n!},
  \]
  where \( a_n \) is the coefficient of \( x_1x_2\cdots x_n \) in \( f(\vx) \). In particular, 
  we have
  \( \theta(e_n(\vx))=x^n/n!\),
  which implies
\begin{equation}\label{eq:theta(c_i)}
  \theta(f_n(\vx))=I_n(2x)=\sum_{r\geq 0}\binom{|n|+2r}{r}\frac{x^{|n|+2r}}{(|n|+2r)!}.
\end{equation}
Moreover, by definition, if \( \lambda\vdash n \), we have
\begin{equation}
  \label{eq:theta(s)}
  \theta(s_\lambda(\vx)) = f^{\lambda}\frac{x^n}{n!},
\end{equation}
where \( f^{\lambda} \) denotes the number of SYTs of shape \( \lambda \).
If, with this knowledge, one applies~$\theta$ to both sides
of~\eqref{eq:BK_odd1} and~\eqref{eq:BK_even1}, then one
obtains~\eqref{eq:Gessel1} and~\eqref{eq:Gessel2}, respectively.

The property of the operator~$\theta$ that is most relevant in our
context is that \( \theta \) intertwines \( \pop \) and the
differential operator \( d/dx \). More precisely, for \( f(\vx)\in \Lambda
\), we have
\begin{equation}\label{eq:pop=d/dX}
  \theta(\pop f(\vx)) = \frac{d}{dx} \theta (f(\vx)),
\end{equation}
as is not very hard to check.

\begin{proof}[Proof of Theorem~\ref{thm:KLO}]
  By \eqref{eq:theta(c_i)}, \eqref{eq:theta(s)} and~\eqref{eq:pop=d/dX},
  application of \( \theta \) to Theorem~\ref{thm:BK2} yields that
$$
    \sum_{n\ge0} |\SYT_{{n},2w+1}| \frac{x^n}{n!}
    = \exp(x)\det_{1\le i,j \leq w}\left(\left( \frac{d}{dx}
    \right)^{i+j-2} (I_{0}(2x)-I_{2}(2x))\right),
$$
and
$$
    \sum_{n\ge0} |\SYT_{{n},2w}| \frac{x^n}{n!}
    = \det_{1\le i,j \leq w}\left(\left( \frac{d}{dx} \right)^{i+j-2} (I_{0}(2x)+I_{1}(2x))\right).
$$
  From the fact that
  \[
    I_{0}(2x)-I_{2}(2x)= \sum_{n\ge0}  \Cat(n/2) \frac{x^n}{n!} \qquad \text{and} \qquad
    I_{0}(2x)+I_{1}(2x)= \sum_{n\ge0}  \binom{n}{\flr{n/2}} \frac{x^n}{n!},
  \]
extraction of coefficients of $\frac {x^n} {n!}$ on both sides of the
above equations finishes the proof.
\end{proof}

\begin{proof}[Proof of Theorem~\ref{cor:ref_KLO}]
The first identity \eqref{eqn:ref_KLOoddSYT} can be proved similarly
as in the proof of Theorem~\ref{thm:KLO}.
For the second identity \eqref{eqn:ref_KLOevenSYT},
we use the fact that for a positive integer \( \alpha \),
\[
 I_{\alpha-1}(2x) - I_{\alpha+1}(2x) = \sum_{s \ge 0} F(s+\alpha-1, s/2) \frac{x^{s+\alpha-1}}{(s+\alpha-1)!}.
\]
By applying \( \theta \) to~\eqref{eq:evenref_k}, we obtain
\begin{align*}
  \sum_{n\geq 0}\left( \underset {\lambda_1\le 2w,\ r(\lambda)=k}
  {\sum_{\lambda\vdash n}} f^{\lambda} \right) \frac{x^n}{n!}
&= \det_{1\le i,j \leq w}\left(\left( \frac{d}{dx} \right)^{j-1} (I_{i+k\delta_{i,w}-1}(2x)-I_{i+k\delta_{i,w}+1}(2x))\right)\\
&= \det_{1\le i,j \leq w}\left( \sum_{s\ge0} F(i+k\delta_{i,w}+s-1,s/2) \frac{x^{i+k\delta_{i,w}+s-j}}{(i+k\delta_{i,w}+s-j)!} \right).
\end{align*}
The proof is completed by extracting the coefficient of \(
x^{t_i}/t_i! \) in the $i$-th row, $1\le i\le w$, of the determinant
on the right-hand side, where $t_1+t_2+\dots+t_w=n$.
\end{proof}

\section{Combinatorial interpretations using up-down tableaux}
\label{sec:comb-interpr}

In this section, we give combinatorial interpretations for the
right-hand sides of the identities in
Theorems~\ref{thm:Goulden} and \ref{thm:row_Goulden} in terms of
(marked) up-down tableaux.
Our main results are Theorems~\ref{thm:Goulden_MUDeven},
\ref{thm:Goulden_MUDodd}, and~\ref{thm:UD-tab}.

\subsection{Basic definitions and preliminaries}
We begin with basic definitions and auxiliary results.
We follow the notation in \cite[Section~7]{HKKO}.

\begin{defn} \label{def:UD}
  Let \( \mu \) and \( \nu \) be partitions. A
  \emph{\( w \)-up-down tableau} \( T \) of length \( 2n \) from \( \mu \)
  to \( \nu \) is a sequence
  \( (T_0,T_1,\dots,T_{2n}) \) of partitions satisfying
  the following properties:
  \begin{enumerate}
  \item [(i)]  \(
    \mu=T_0\subseteq T_1\supseteq T_2\subseteq
     T_3\supseteq T_4\subseteq\dots\subseteq T_{2n-1}\supseteq
     T_{2n}=\nu
    \);
  \item [(ii)]each pair $( T_{i-1}, T_i)$ differs by a vertical strip
    {\rm(}that is, by a collection of cells which contains at most one
    cell in each row{\rm)}, for $i=1,2,\dots,2n$;
  \item [(iii)]each $ T_{i}$ has at most $w$~rows, for
    $i=0,1,\dots,2n$.
  \end{enumerate}
  The weight of \( T \) is defined by 
  \[
  \omega(T)=\prod_{i=1}^{n} x_i^{-| T_{2i-2}|+2| T_{2i-1}|-| T_{2i}|}.
\]
It should be noted that the exponent
$-| T_{2i-2}|+2| T_{2i-1}|-| T_{2i}|$ is the sum of
of the differences in sizes of
$( T_{2i-2}, T_{2i-1})$ and of $( T_{2i-1}, T_{2i})$,
for $i=1,2,\dots,n$.

We denote the set of
\( w \)-up-down tableaux of length \( 2n \) from \( \mu \) to \( \nu
\) by \( \UD_n(w;\mu\to \nu) \). 
  
\end{defn}

We will identify
\( T = ( T_0,  T_1,\dots, T_{2n}) \in \UD_n(w;\mu\to \nu) \) with the
family of nonintersecting paths \( \mathbf{P}=(P_1,\dots,P_w) \),
where \( P_i \) is the path from 
\( (\mu_{w+1-i}+i,0) \) to
\( (\nu_{w+1-i}+i,2n) \) consisting of the points
\( ((T_j)_{w+1-i}+i,j) \) for \( j=0,1,\dots,2n \). 
For example, the sequence
\begin{multline*}
((1,1,0,0),\
(1,1,0,0),\
(1,0,0,0),\
(2,1,1,1),\
(2,1,1,0),\
(3,2,2,0),\\
(2,2,1,0),\
(2,2,1,1),\
(1,1,1,1),\
(1,1,1,1),\
(0,0,0,0))
\end{multline*}
forms a \( 4 \)-up-down tableau in \( \UD_5(4;(1,1,0,0)\to (0,0,0,0)) \) 
and is identified with the family \( \mathbf{P} \) in Figure~\ref{fig:paths}.

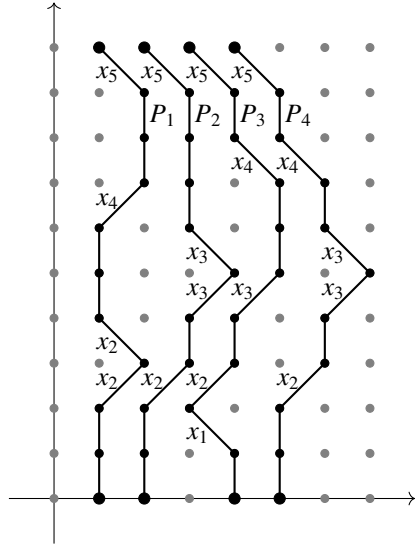
\begin{figure}

\small{
\begin{tikzpicture}[scale=0.6]

\draw[->] (-1,0) -- (8,0);
\draw[->] (0,-1) -- (0,11);

\foreach \i in {0,...,7}
\foreach \j in {0,...,10}
\filldraw[fill=gray, color=gray] (\i,\j) circle (2.5pt);

\foreach \i in {1,2,4,5}
\filldraw (\i,0) circle (3.5pt);

\foreach \i in {1,2,4,5}
\filldraw (\i,1) circle (2.5pt);

\foreach \i in {1,2,3,5}
\filldraw (\i,2) circle (2.5pt);

\foreach \i in {2,3,4,6}
\filldraw (\i,3) circle (2.5pt);

\foreach \i in {1,3,4,6}
\filldraw (\i,4) circle (2.5pt);

\foreach \i in {1,4,5,7}
\filldraw (\i,5) circle (2.5pt);

\foreach \i in {1,3,5,6}
\filldraw (\i,6) circle (2.5pt);

\foreach \i in {2,3,5,6}
\filldraw (\i,7) circle (2.5pt);

\foreach \i in {2,3,4,5}
\filldraw (\i,8) circle (2.5pt);

\foreach \i in {2,3,4,5}
\filldraw (\i,9) circle (2.5pt);

\foreach \i in {1,2,3,4}
\filldraw (\i,10) circle (3.5pt);

\node[right] at (1.9,8.5) {\( P_1\)};
\node[right] at (2.9,8.5) {\( P_2\)};
\node[right] at (3.9,8.5) {\( P_3\)};
\node[right] at (4.9,8.5) {\( P_4\)};

\foreach \i/\j in {3.65/1.4}
\node[left] at (\i,\j) {\(x_1\)};

\foreach \i/\j in {1.65/2.6, 2.65/2.6, 3.65/2.6, 5.65/2.6,1.65/3.4}
\node[left] at (\i,\j) {\(x_2\)};

\foreach \i/\j in {3.65/4.6, 4.65/4.6, 6.65/4.6,3.65/5.4, 6.65/5.4}
\node[left] at (\i,\j) {\(x_3\)};

\foreach \i/\j in {1.65/6.6, 4.65/7.4, 5.65/7.4}
\node[left] at (\i,\j) {\(x_4\)};

\foreach \i/\j in {1.65/9.4, 2.65/9.4, 3.65/9.4, 4.65/9.4}
\node[left] at (\i,\j) {\(x_5\)};

\draw[thick] (1,0) -- (1,2) -- (2,3) -- (1,4) -- (1,6) -- (2,7) -- (2,9) -- (1,10)
(2,0) -- (2,2) -- (3,3) -- (3,4) -- (4,5) -- (3,6) -- (3,9) -- (2,10)
(4,0) -- (4,1) -- (3,2) -- (4,3) -- (4,4) -- (5,5) -- (5,7) -- (4,8) -- (4,9) -- (3,10)
(5,0) -- (5,1) -- (5,2) -- (6,3) -- (6,4) -- (7,5) -- (6,6) -- (6,7) -- (5,8) -- (5,9) -- (4,10);

\end{tikzpicture}

}
\caption{An example of a family \(\mathbf{P}=(P_1,P_2,P_3,P_4)\) of nonintersecting paths.}
\label{fig:paths}
\end{figure}

\begin{defn}
  Let \( L(u\to v) \) denote the set of lattice paths from \( u \) to
\( v \) with steps from the set 
\[
S=\{ (i,j) \rightarrow (i,j+1), ~(i,2j-2) \rightarrow (i+1,2j-1), ~ (i,2j-1) \rightarrow (i-1,2j) : i,j\in \ZZ \}.
\]
We define the weight of a vertical step \( (i,j) \rightarrow (i,j+1) \) to be 1 and the weight of a forward
(respectively backward) diagonal step \( (i,2j-2) \rightarrow (i+1,2j-1) \) 
(respectively \( (i,2j-1) \rightarrow (i-1,2j) \)) to be \( x_j \). 
The weight of a path \( P \) equals the product of the weights of its
steps.
For example, the weight of the left-most path in
Figure~\ref{fig:paths} equals $x_2^2x_4x_5$.

Let \( L(a,b;u\to v) \) denote the set of paths
\( P\in L(u\to v) \) for which every point \( (i,2j) \) of even height
in \( P \) satisfies \( a\le i\le b \). We also write
\( L_{t}(u\to v) = L(t,\infty;u\to v) \).
\end{defn}

\begin{defn}
  An \emph{odd-branch point} of a lattice path \( P \) is a point
  \( (1,2j-1) \) with the property that \( P \) passes through
  \( (1,2j-2)\), \((1,2j-1) \), and \( (1,2j) \).
  An \emph{even-branch point} of a lattice path \( P \) is a point
  \( (1,2j-2) \) with the property that \( P \) passes through
  \( (1,2j-2)\), \((2,2j-1) \), and \( (2,2j) \).
  See Figure~\ref{fig:image1}.

\begin{figure}
\begin{tikzpicture}[scale=0.6]

\draw[color=gray] (0,0) -- (0,2);
\foreach \i in {0,1,2}
\foreach \j in {0,1,2}
\filldraw[fill=gray, color=gray] (\i,\j) circle (2.5pt);

\node[left] at (0,0) {\( 2j-2 \)};
\node[left] at (0,1) {\( 2j-1 \)};
\node[left] at (0,2) {\( 2j \)};
\node[below] at (1,-0.5) {\( x=1 \)};
\node[right] at (1,1) {\( u \)};

\foreach \i in {0,1,2}
\filldraw (1,\i) circle (2.5pt);

\draw[thick] 
(1,0) -- (1,1) -- (1,2) ;

\draw[dashed] (1,-.5) -- (1,2.5);

\end{tikzpicture}
\qquad \qquad \qquad \qquad
\begin{tikzpicture}[scale=0.6]

\draw[dashed] (1,-.5) -- (1,2.5);
\draw[color=gray] (0,0) -- (0,2);
\foreach \i in {0,1,2}
\foreach \j in {0,1,2}
\filldraw[fill=gray, color=gray] (\i,\j) circle (2.5pt);

\node[left] at (0,0) {\( 2j-2 \)};
\node[left] at (0,1) {\( 2j-1 \)};
\node[left] at (0,2) {\( 2j \)};
\node[below] at (1,-0.5) {\( x=1 \)};
\node[right] at (1,0) {\( v \)};

\foreach \i in {1,2}
\filldraw (2,\i) circle (2.5pt);

\filldraw (1,0) circle (2.5pt);

\draw[thick] 
(1,0) -- (2,1) -- (2,2) ;

\end{tikzpicture}

\caption{An odd-branch point \( u \) (left) and an even-branch point \( v \) (right).}
\label{fig:image1}
\end{figure}
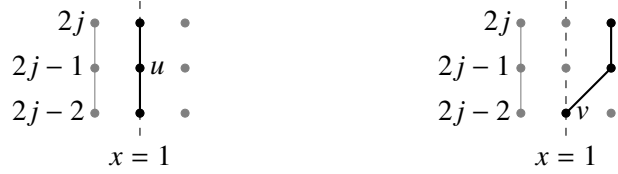

For \( t\ge1 \), an \emph{odd-marked lattice path}
(respectively~\emph{even-marked lattice path}) is a pair \( (P,M) \) of a
lattice path \( P\in L_t(u\to v) \) and a set \( M \) of integers such
that if \( j\in M \), then \( (1,2j-1) \) (respectively~\( (1,2j-2) \)) is an
odd-branch (respectively~even-branch) point of \( P \). Let
\( L_t^o(u\to v) \) (respectively~\( L_t^e(u\to v) \)) denote the set of
odd-marked (respectively~even-marked) lattice paths \( (P,M) \) with
\( P\in L_t(u\to v) \).
\end{defn}

The lemma below collects basic generating functions for various
sets of lattice paths that we shall need in the sequel.

\begin{lem}\label{lem:3}
  For positive integers \( i \), \( j \), and \( n \), we have
  \begin{align}
    \label{eq:UD-lem0}
    f_{i-j}(x_1,\dots,x_n)-f_{i+j}(x_1,\dots,x_n)
    &= \sum_{P\in L_1((i,0)\to (j,2n))} \omega(P),\\
    \label{eq:UD-lem2}
    f_{i-j}(x_1,\dots,x_n)-f_{i+j-1}(x_1,\dots,x_n)
    &= \sum_{(P,M)\in L^o_1((i,0)\to (j,2n))} (-1)^{|M|} \omega(P) \omega(M),\\
    \label{eq:UD-lem1}
    f_{i-j}(x_1,\dots,x_n)+f_{i+j-1}(x_1,\dots,x_n)
    &= \sum_{(P,M)\in L^o_1((i,0)\to (j,2n))} \omega(P) \omega(M),\\
    \label{eq:UD-lem3}
    f_{i-j}(x_1,\dots,x_n)-f_{i+j-2}(x_1,\dots,x_n)
    &= \sum_{P\in L_2((i,0)\to (j,2n))} \omega(P),\\
    \label{eq:UD-lem4}
    f_{i-j}(x_1,\dots,x_n)+f_{i+j-2}(x_1,\dots,x_n)
    &= \sum_{(P,M)\in L^e_1((i,0)\to (j,2n))} 2^{\chi(j=1)}\omega(P),
  \end{align}
  where \( \omega(M) = \prod_{j\in M} x_j  \).
\end{lem}

\begin{proof}
  Identity~\eqref{eq:UD-lem0} is obtained from Equation~(7.3)
  in~\cite{HKKO} by taking the limit \( N\to \infty \) and using the
  fact \( f_{-i+j} = f_{i-j} \). Similarly, Equation~\eqref{eq:UD-lem1} is
  obtained from Equation~(7.11) in~\cite{HKKO}.
  Identity~\eqref{eq:UD-lem3} is obtained from \eqref{eq:UD-lem0} with
  \( i \) and \( j \) replaced by \( i-1 \) and \( j-1 \),
  respectively. Identity~\eqref{eq:UD-lem2} follows
  from~\eqref{eq:UD-lem1} by replacing every \( x_t \) by \( -x_t \).

  Identity \eqref{eq:UD-lem4} can be proved similarly as in the
  proof of~\cite[Eq.~(7.3)]{HKKO}. The idea is that we find the smallest
  \( t \) such that the path \( P \) visits \( (1,2t-2),(1,2t-1) \),
  and \( (0,2t) \). Let \( P=P'P'' \), where \( P' \) and \( P'' \) are
  the subpaths of \( P \) obtained by dividing it at \( (1,2t-2) \).
  Update \( P \) to be the path \( P'\mathfrak{R}(P'') \) and add
  \( t \) to \( M \), where \( \mathfrak{R} \) is the map in
  \cite[proof of Lemma~7.3]{HKKO}. Repeat this process until there is
  no such \( t \). Then we obtain
  \( (P,M) \in L^e_1((i,0)\to (j,2n))\).
\end{proof}

\subsection{Combinatorial interpretation for \Cref{thm:Goulden}}

Our combinatorial interpretation of the right-hand side
of~\eqref{eq:Goulden2m,k} is the following. 

\begin{thm} \label{thm:Goulden_MUDeven}
  For integers \( 0\le k\le w \), we have
  \begin{align}
    \label{eq:UD-Gouldeneven1}
    \underset {r(\lambda) = k}
    {\sum_{\lambda:\lambda_1 \le 2w}} s_{\lambda}(\vx_n)
    = \sum_{T\in \UD_n(w; (k,0^{w-1}) \to (0^w))}
   \omega(T).  
   \end{align}
\end{thm}

\begin{proof} 
This is in fact the special case of a more general result; cf.\
\cite[Eq.~(1.1)]{KratAV}. For the convenience of the reader,
we provide the (specialized) argument. Namely,
by \eqref{eq:Goulden2m,k} and \eqref{eq:UD-lem0}, we have
  \begin{equation}\label{eq:evenr}
\underset{r(\lambda)=k}
    {\sum_{\lambda:\lambda_1 \le 2w}} s_{\lambda}(\vx_n)
    =\det_{1\le i,j\le w}
      \left( f_{i+k\delta_{i,w}-j}(\vx_n)-f_{i+k\delta_{i,w}+j}(\vx_n) \right)
 = \sum_{\mathbf{P}\in X} \sgn(\mathbf{P}) \prod_{i=1}^{w} \omega(P_i),
  \end{equation}
  where \( X \) is the set of families
  \( \mathbf{P}=(P_1,\dots,P_w) \) of lattice paths
  \( P_i \in L_1((i+k \delta_{i,w},0)\to (\sigma(i),2n)) \) for some
  permutation \( \sigma \) on \( [w] \), and
  \( \sgn(\mathbf{P}) = \sgn(\sigma) \). By applying the
  Lindstr\"om--Gessel--Viennot lemma \cite[Lemma~1]{LindAA}
  to~\eqref{eq:evenr}, we have
  \begin{equation}\label{eq:evenr2}
    \det_{1\le i,j\le w}
      \left( f_{i+k\delta_{i,w}-j}(\vx_n)-f_{i+k\delta_{i,w}+j}(\vx_n) \right)
 =  \sum_{\mathbf{P}\in Y}  \prod_{i=1}^{w}  \omega(P_i),
  \end{equation}
  where \( Y \) is the set of families
  \( \mathbf{P}=(P_1,\dots,P_w) \) of nonintersecting lattice paths
  \( P_i \in L_1((i+k\delta_{i,w},0)\to (i,2n)) \). Therefore,
  by~\eqref{eq:Goulden2m,k}, \eqref{eq:evenr2}, and by identifying
  \( \mathbf{P} \) as a \( w \)-up-down tableau \( T \), we
  obtain
  \begin{align}
    \label{eq:UD-Gouldeneven1-1}
    \underset {r(\lambda) = k}
    {\sum_{\lambda:\lambda_1 \le 2w}} s_{\lambda}(\vx_n)
    = \sum_{T\in \UD_n(w; (0^w) \to (k,0^{w-1}))}
   \omega(T).  
   \end{align}
   By the symmetry of the variables \( x_1,\dots,x_n \),
   \eqref{eq:UD-Gouldeneven1-1} is equivalent to
   \eqref{eq:UD-Gouldeneven1}.
\end{proof}

For our upcoming combinatorial interpretation of the right-hand side
of~\eqref{eq:Goulden2m+1,k} we need one more definition.

\begin{defn} \label{def:UD2}
  Let \( \MUD_n(w;\mu\to \nu) \) denote the set of pairs \( (T,S) \) such
  that \( T\in \UD_n(w;\mu\to \nu) \) and \( S \subseteq [n] \).
  We call an element \( (T,S)\in \MUD_n(w;\mu\to \nu) \)
  a \emph{marked \( w \)-up-down tableau}. 
\end{defn}

Here is now our combinatorial interpretation of the right-hand side
of~\eqref{eq:Goulden2m+1,k}.

\begin{thm}\label{thm:Goulden_MUDodd}
  For integers \( 0\le k\le w \), we have
  \begin{align}
    \label{eq:UD-Gouldenodd1}
    \underset {r(\lambda) = k}
    {\sum_{\lambda:\lambda_1 \le 2w+1}} s_{\lambda}(\vx_n)
    = \underset {|S|=k}
    {\sum_{{(T,S)\in \MUD_n(w; (0^w) \to (0^w))}}}
   \omega(T) \omega(S),  
   \end{align}
   where \( \omega(S)=\prod_{j\in S} x_j \).
\end{thm}

\begin{proof} 
By \eqref{eq:UD-lem0}, we have
 \begin{equation}\label{eq:oddc}
   \det_{1 \le i, j \le w} \left(
   f_{i-j}(\vx_n) - f_{i+j}(\vx_n) \right) \\
 = \sum_{\mathbf{P}\in X} \sgn(\mathbf{P}) \prod_{i=1}^{w} \omega(P_i),
  \end{equation}
  where \( X \) is the set of families \( \mathbf{P}=(P_1,\dots,P_w) \) of lattice paths
  \( P_i \in L_1((i,0)\to (\sigma(i),2n)) \)
  for some permutation \( \sigma \) on \( [w] \), and
  \( \sgn(\mathbf{P}) = \sgn(\sigma) \).
  By applying the Lindstr\"om--Gessel--Viennot lemma
  \cite[Lemma~1]{LindAA} to~\eqref{eq:oddc}
  and multiplying \( e_k(\vx_n) \) on both sides, we have
  \begin{equation}\label{eq:oddc2}
   e_k(\vx_n) \det_{1 \le i, j \le w} \left(
   f_{i-j}(\vx_n) - f_{i+j}(\vx_n) \right) \\
 =  \sum_{(\mathbf{P},S)\in Y}  \omega(S) \prod_{i=1}^{w}  \omega(P_i),
  \end{equation}
  where \( Y \) is the set of \( (\mathbf{P},S) \) such that
  \( \mathbf{P}=(P_1,\dots,P_w) \) is a family of nonintersecting
  lattice paths \( P_i \in L_1((i,0)\to (i,2n)) \) and
  \( S \subseteq [n] \) with \( |S|=k \). Therefore,
  by~\eqref{eq:Goulden2m+1,k}, \eqref{eq:oddc2}, and by identifying
  \( \mathbf{P} \) as a \( w \)-up-down tableau \( T \), we
  obtain~\eqref{eq:UD-Gouldenodd1}.
\end{proof}

\subsection{Combinatorial interpretation for \Cref{thm:row_Goulden} (odd bound case)}

The goal of the remainder of this section is to provide combinatorial
interpretations of the right-hand sides of the identities in
Theorem~\ref{thm:row_Goulden}, in its equivalent form given in
Theorem~\ref{thm:row_Goulden2-2}. 

\begin{defn} \label{def:UD3}
  Let \( T\in \UD_n(w;\mu\to \nu) \). For an integer \( j\ge 1 \), we say
  that \( 2j-1 \) is a \emph{length-peak} of \( T \) if
  \( \ell(T_{2j-2}) < \ell(T_{2j-1}) > \ell(T_{2j}) \).
  Furthermore, we say
  that the length-peak \( 2j-1 \) is \emph{full} if
  \( \ell(T_{2j-1})=w \), and \emph{non-full} otherwise. We define
  \[
    \widehat{\UD}_n(w;\mu\to \nu)
    = \{ T\in \UD_n(w;\mu\to \nu): \mbox{every length-peak of \( T \), if any, is full} \}.
  \]
\end{defn}

\begin{defn}\label{def:MUD^o} \label{def:UD4}
    Let 
  \( \MUD_n^o(w;\mu\to \nu) \) denote the set of pairs \( (T,S) \) of
  \( T\in \widehat{\UD}_n(w;\mu\to \nu) \) and \( S \subseteq [n] \) satisfying the following conditions:
  \begin{itemize} 
  \item if \( 2j-1 \) is a length-peak of \( T \), then \( j\in S \);
  \item if \( \ell(T_{2j-1}) <w \), then \( j\not\in S \). 
  \end{itemize}
\end{defn}

The following lemma concerns the right-hand sides
of~\eqref{eq:schur-odd-k} and~\eqref{eq:schur-odd-2h+1-k}, that is,
the identities with an odd bound on the number of columns of the
partitions over which the sum is taken on the left-hand side.

\begin{lem}\label{lem:1}
  We have
 \begin{align}
   \label{eq:9}
   e(\vx_n) \det_{1 \le i, j \le w} \left(
   f_{i+\chi(i>w-k)-j}(\vx_n) - f_{i+\chi(i>w-k)+j-1}(\vx_n) \right)
   &= \sum_{(T,S)\in \MUD_n^o(w; (1^k,0^{w-k}) \to (0^w))} \omega(T) \omega(S) ,\\
   \label{eq:14}
   \overline{e}(\vx_n) \det_{1 \le i, j \le w} \left(
   f_{i+\chi(i>w-k)-j}(\vx_n) + f_{i+\chi(i>w-k)+j-1}(\vx_n) \right)
   &= \sum_{(T,S)\in \MUD_n^o(w; (1^k,0^{w-k}) \to (0^w))} (-1)^{|S|} \omega(T) \omega(S),
 \end{align}
 where \( \omega(S)=\prod_{j\in S} x_j \).
\end{lem}

\begin{proof}
  By \eqref{eq:UD-lem2}, we have
  \begin{equation}\label{eq:15}
   \det_{1 \le i, j \le w} \left(
   f_{i+\chi(i>w-k)-j}(\vx_n) - f_{i+\chi(i>w-k)+j-1}(\vx_n) \right) \\
 = \sum_{\mathbf{Q}\in X} \sgn(\mathbf{Q}) \prod_{i=1}^{w} (-1)^{|M_i|} \omega(P_i) \omega(M_i),
  \end{equation}
  where \( X \) is the set of tuples \( \mathbf{Q}=(Q_1,\dots,Q_w) \) of
  odd-marked lattice paths
  \[
    Q_i=(P_i,M_i) \in L_1^o((i+\chi(i>w-k),0)\to (\sigma(i),2n)),
  \]
  for some permutation \( \sigma \) on \( [w] \), and
  \( \sgn(\mathbf{Q}) = \sgn(\sigma) \).
  By applying the Lindstr\"om--Gessel--Viennot lemma
  \cite[Lemma~1]{LindAA} to~\eqref{eq:15}
  and multiplying \( e(\vx_n) \) on both sides, we have
  \begin{equation}\label{eq:18}
   e(\vx_n) \det_{1 \le i, j \le w} \left(
   f_{i+\chi(i>w-k)-j}(\vx_n) - f_{i+\chi(i>w-k)+j-1}(\vx_n) \right) \\
 =  \sum_{(\mathbf{Q},S)\in Y}  \omega(S) \prod_{i=1}^{w} (-1)^{|M_i|} \omega(P_i) \omega(M_i),
  \end{equation}
  where \( Y \) is the set of \( (\mathbf{Q},S) \)
  such that \( \mathbf{Q}=(Q_1,\dots,Q_w) \) is a family of 
  nonintersecting odd-marked lattice paths
  \[
    Q_i=(P_i,M_i) \in L_1^o((i+\chi(i>w-k),0)\to (i,2n))
  \]
  and \( S \subseteq [n] \). Observe that, if \( (\mathbf{Q},S)\in Y \) with
  \( Q_i=(P_i,M_i) \), then \( M_i = \emptyset \) for all
  \( i\ne 1 \).

  Suppose that \( (\mathbf{Q},S)\in Y \). Let \( j \) be the smallest integer
  such that \( P_1 \) contains \( (1,2j-1) \) and
  \( j\in (M_1 \cup S)\setminus (M_1 \cap S) \). If such \( j \)
  exists, let \( (\mathbf{Q}',S') \) be the pair obtained from \( (\mathbf{Q},S) \) by
  moving the integer \( j \) from \( M_1 \) to \( S \) or vice versa,
  otherwise let \( (\mathbf{Q}',S') = (\mathbf{Q},S) \). Then the map
  \( (\mathbf{Q},S)\mapsto (\mathbf{Q}',S') \) is a sign-reversing involution on \( Y \)
  whose fixed points are the pairs \( (\mathbf{Q},S)\in Y \) with \( Q_i=(P_i,M_i) \)
  such that, if \( P_1 \) contains \( (1,2j-1) \), then
  \( j\in M_1\cap S \) or \( j\not\in M_1\cup S \), and if \( P_1 \)
  does not contain \( (1,2j-1) \), then \( j\not\in M_1 \). Observe
  that, if \( (\mathbf{Q},S) \) is a fixed point, then \( M_1 \) is completely
  determined by \( P_1 \) and \( S \). 
  Therefore, by identifying
  \( (P_1,\dots,P_w) \) as a \( w \)-up-down tableau \( T \), we may
  rewrite~\eqref{eq:18} as
  \begin{equation}\label{eq:19}
   e(\vx_n) \det_{1 \le i, j \le w} \left(
   f_{i+\chi(i>w-k)-j}(\vx_n) - f_{i+\chi(i>w-k)+j-1}(\vx_n) \right) \\
 = \sum_{(T,S)\in \MUD_n(w; (1^k,0^{w-k}) \to (0^w))}
\omega(T) \prod_{j\in S} x_j(-x_j )^{\chi(\ell(T_{2j-1})<w)}.
  \end{equation}

  We will further simplify the right-hand side of~ \eqref{eq:19} by
  finding a sign-reversing involution. Let
  \( (T,S)\in \MUD_n(w; (1^k,0^{w-k}) \to (0^w)) \). Find the smallest
  \( j \) satisfying one of the following conditions:
\begin{description}
\item[\sc Case 1] \( j\in S \) and \( \ell(T_{2j-1})<w \). In this case,
  let \( T' \) be the up-down tableau obtained from \( T \) by
  creating a new cell at the end of row one of \( T_{2j-1} \), and let
  \( S'=S \setminus \{j\} \). Then \( 2j-1 \) is a length-peak of the
  resulting up-down tableau \( T' \) and \( j\notin S' \).
\item[\sc Case 2] \( j\notin S \) and \( 2j-1 \) is a length-peak of
  \( T \). In this case, let \( T' \) be the up-down tableau obtained
  from \( T \) by deleting the last row, which has one cell, from
  \( T_{2j-1} \), and let \( S'=S \cup \{j\} \). Then
  \( \ell(T'_{2j-1})<w \) and \( j\in S' \).
\end{description}
See \Cref{fig:1}. If there is no such \( j \), define \( T'=T \) and
\( S'=S \). Then the map \( (T,S) \mapsto (T',S') \) is a
sign-reversing involution on \( \MUD_n(w; (1^k,0^{w-k}) \to (0^w)) \)
with fixed point set \( \MUD_n^o(w; (1^k,0^{w-k}) \to (0^w)) \). This
shows the first identity, Equation~\eqref{eq:9}.

\medskip
The second identity, Equation~\eqref{eq:14}, can be proved similarly, where the
following equation plays the role of~\eqref{eq:19}:
\begin{multline*}
     \overline{e}(\vx_n) \det_{1 \le i, j \le w} \left(
   f_{i+\chi(i>w-k)-j}(\vx_n) + f_{i+\chi(i>w-k)+j-1}(\vx_n) \right)\\
 = \sum_{(T,S)\in \MUD_n(w; (1^k,0^{w-k}) \to (0^w))}
\omega(T) \prod_{j\in S} (-x_j) (-x_j )^{\chi(\ell(T_{2j-1})<w)}.
\qedhere
\end{multline*}
\end{proof}

\begin{figure}
\begin{tikzpicture}[scale=0.6]

\draw[color=gray] (0,0) -- (0,2);
\foreach \i in {0,1,2}
\foreach \j in {0,1,2}
\filldraw[fill=gray, color=gray] (\i,\j) circle (2.5pt);

\node[left] at (0,0) {\( 2j-2 \)};
\node[left] at (0,1) {\( 2j-1 \)};
\node[left] at (0,2) {\( 2j \)};
\node[below] at (4,-0.5) {\( x=t \)};
\node[right] at (5,1) {\( u \)};
\node[below] at (2,-1.5) {\( j \in S, ~t\geq 1,~ u\not\in \mathbf{P}\) };
\node[right] at (7,1) {\( \Leftrightarrow \)};

\foreach \i in {1,2,3,4}
\foreach \j in {0,1,2}
\filldraw (\i,\j) circle (2.5pt);

\filldraw[fill=gray, color=gray] (5,1) circle (2.5pt);

\foreach \i in {1,2,3,4}
\draw[thick] (\i,0) -- (\i,1) -- (\i,2);

\draw[dashed] (4,-.5) -- (4,2.5);

\end{tikzpicture}
\qquad  
\begin{tikzpicture}[scale=0.6]

\draw[dashed] (4,-.5) -- (4,2.5);
\draw[color=gray] (0,0) -- (0,2);
\foreach \i in {0,1,2}
\foreach \j in {0,1,2}
\filldraw[fill=gray, color=gray] (\i,\j) circle (2.5pt);

\node[left] at (0,0) {\( 2j-2 \)};
\node[left] at (0,1) {\( 2j-1 \)};
\node[left] at (0,2) {\( 2j \)};
\node[below] at (4,-0.5) {\( x=t \)};
\node[below] at (2,-1.5) {\( j \not\in S, ~t\geq 1 \) };

\foreach \i in {1,2,3,4}
\foreach \j in {0,1,2}
\filldraw (\i,\j) circle (2.5pt);

\filldraw[fill=gray, color=gray] (4,1) circle (2.5pt);
\filldraw (5,1) circle (2.5pt);

\foreach \i in {1,2,3}
\draw[thick] (\i,0) -- (\i,1) -- (\i,2);

\draw[thick] (4,0) -- (5,1) -- (4,2);
\end{tikzpicture}

\caption{Configurations for the involution in Lemma~\ref{lem:1}.}
\label{fig:1}
\end{figure}
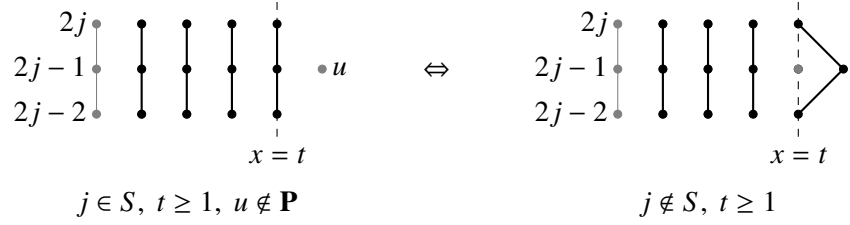

We are now in the position to state, and prove, our combinatorial
interpretations of
the sums of Schur functions over partitions with an odd bound on the
number of columns and a given number of odd columns appearing
in~\eqref{eq:schur-odd-k}
and~\eqref{eq:schur-odd-2h+1-k}.

\begin{thm}\label{thm:1}
  For integers \( 0\le k\le w \), we have
  \begin{align}
    \label{eq:UD-odd1}
    \underset {c(\lambda) = k}
    {\sum_{\lambda:\lambda_1 \le 2w+1}} s_{\lambda}(\vx_n)
    &=\underset {|S|\text{ \em even}}
    {\sum_{{(T,S)\in \MUD_n^o(w; (1^k,0^{w-k}) \to (0^w))}}}
   \omega(T) \omega(S),\\
    \label{eq:UD-odd2}
    \underset {c(\lambda) = 2w+1-k}
    {\sum_{\lambda:\lambda_1 \le 2w+1}} s_{\lambda}(\vx_n)
    &= \underset {|S|\text{ \em odd}}
    {\sum_{{(T,S)\in \MUD_n^o(w; (1^k,0^{w-k}) \to (0^w))}}}
   \omega(T) \omega(S).
  \end{align}
\end{thm}

\begin{proof}
  These identities are obtained by adding and subtracting the two
  equations~\eqref{eq:schur-odd-k} and~\eqref{eq:schur-odd-2h+1-k}, and applying the two 
  equations in \Cref{lem:1}.
\end{proof}

\subsection{Combinatorial interpretation for \Cref{thm:row_Goulden} (even bound case)}

We finally turn our attention to the ``even bound identities," that
is, to~\eqref{eq:schur-even-k} and~\eqref{eq:schur-even-2h-k}.

Recall that \( \widehat{\UD}_n(w;\mu\to \nu) \) is the set of
\( T\in \UD_n(w;\mu\to \nu) \) without any non-full length-peaks.

\begin{defn} \label{def:UD5}
  For \( T\in \UD_n(w;\mu\to \nu) \), 
  let 
  \[
    E_w(T) = \{ j \in [n]:  \ell(T_{2j-2}) < \ell(T_{2j-1}) =  \ell(T_{2j}) = w \}.
  \]
  We define 
  \begin{align}
    \label{eq:MUD*}
    \MUD_n^*(w;\mu\to \nu)
    &= \{(T,S): T\in \widehat{\UD}_n(w;\mu\to \nu) \mbox{ and } S \subseteq E_w(T)\},\\
    \label{eq:MUD<}
    \MUD_n^<(w;\mu\to \nu)
    &= \{(T,S): T\in \UD_n(w;\mu\to \nu), \ell(T_j)<w \mbox{ for all \( j=0,1,\dots,2n \), and } S \subseteq[n]\}.
  \end{align}
\end{defn}

The following lemma concerns the right-hand side
in~\eqref{eq:schur-even-k} and~\eqref{eq:schur-even-2h-k}. 
Here, and also later,
we will use the convention that \( \det A = 1 \) if \( A \) is an
empty matrix.

\begin{lem}\label{lem:5}
  For \( 0 \leq k \leq w \), we have
  \begin{equation}
    \label{eq:1}
    \frac{1}{2} \det_{1 \le i, j \le w} \left(
   f_{i+\chi(i>w-k)-j}(\vx_n) + f_{i+\chi(i>w-k)+j-2}(\vx_n) \right) 
   = \sum_{(T,S)\in  \MUD_n^*(w; (1^k,0^{w-k}) \to (0^w))} \omega(T).
  \end{equation}
  For \( 0 \leq k \leq w-1 \), we have
  \begin{equation}\label{eq:2} 
  e(\vx_n) \overline{e}(\vx_n) 
  \det_{2 \le i, j \le w} \left(
    f_{i+\chi(i>w-k)-j}(\vx_n) - f_{i+\chi(i>w-k)+j-2}(\vx_n)
  \right)
   =  \sum_{(T,S)\in \MUD^<_n(w; (1^k,0^{w-k}) \to (0^w))} \omega(T) \prod_{j\in S}(-x_j^2).
  \end{equation}
\end{lem}

\begin{proof}
  For the first identity, by \eqref{eq:UD-lem4}, we have
  \begin{equation}\label{eq:lemMUD*}
   \frac{1}{2} \det_{1 \le i, j \le w} \left(
   f_{i+\chi(i>w-k)-j}(\vx_n) + f_{i+\chi(i>w-k)+j-2}(\vx_n) \right) 
 = \sum_{\mathbf{Q}\in X} \sgn(\mathbf{Q}) \prod_{i=1}^{w} \omega(P_i),
  \end{equation}
  where \( X \) is the set of tuples \( \mathbf{Q}=(Q_1,\dots,Q_w) \) of
  even-marked lattice paths
  \[
    Q_i=(P_i,M_i) \in L_1^e((i+\chi(i>w-k),0)\to (\sigma(i),2n)),
  \]
  for some permutation \( \sigma \) on \( [w] \), and
  \( \sgn(\mathbf{Q}) = \sgn(\sigma) \).
  Note that the factor \( 1/2 \) accounts for the fact that, if \( j=1 \), then
  \( f_{i+\chi(i>w-k)-j}(\vx_n) = f_{i+\chi(i>w-k)+j-2}(\vx_n) \).

  We apply the Lindstr\"om--Gessel--Viennot lemma \cite[Lemma~1]{LindAA}
  to~\eqref{eq:lemMUD*} as in the proof of~\eqref{eq:9}. However, in
  the present
  case, we must be careful because certain configurations cannot be
  canceled. To be specific, suppose that \( \mathbf{Q} \) contains two
  even-marked paths \( Q_r=(P_r,M_r) \) and \( Q_s=(P_s, M_s) \) such
  that
  \begin{itemize}
  \item \( P_r \) passes the points \( (1,2j-2), (2,2j-1), (2,2j) \), and \( j\in M_r \),
  \item \( P_s \) passes the points \( (2,2j-2), (2,2j-1), (1,2j) \), which forces that \( j\not\in M_s \).
  \end{itemize}
  Let \( P_r' \) and \( P_r'' \) be the subpaths of \( P_r \) before
  and after the point \( u=(2,2j-1) \), respectively. Define
  \( P_s' \) and \( P_s'' \) similarly. Let
  \( \widetilde{P}_r = P'_rP''_s \) and
  \( \widetilde{P}_s = P'_sP''_r \) as shown in \Cref{fig:image2}.
  Since \( (1,2j-2) \) is no longer an even-branch point for
  \( \widetilde{P}_r \) or \( \widetilde{P}_s \), we cannot cancel
  such \( \mathbf{Q} \) by modifying the paths at \( u \). Still, except for
  such an intersection point, we can apply the usual ``tail exchange''
  involution. Thus, we have
  \begin{equation}\label{eq:lemMUD*LGV}
   \frac{1}{2} \det_{1 \le i, j \le w} \left(
   f_{i+\chi(i>w-k)-j}(\vx_n) + f_{i+\chi(i>w-k)+j-2}(\vx_n) \right) 
 =  \sum_{\mathbf{Q}\in Y} \sgn(\mathbf{Q})  \prod_{i=1}^{w}  \omega(P_i) ,
  \end{equation}
  where \( Y \) is the set of \( \mathbf{Q}\in X \) with \( Q_i=(P_i,M_i) \) satisfying the following conditions:
  \begin{itemize} 
  \item \( P_i \) is a path from \( (i+\chi(i>w-k),0) \) to \( (i,2n) \) for \( i=3,\dots, w \),
  \item if there is an intersection point \( u \), then
  \( u=(2,2j-1) \) for some \( j \in M_1\cup M_2 \), \( u \not \in P_i \) for \( i\geq 3 \), 
  and \( P_1 \) visits \( (1,2j-2), u, (1,2j) \), and \( P_2 \) visits \( (2,2j-2), u, (2,2j) \), or vice versa. 
  \end{itemize}
    Observe that, if \( \mathbf{Q}\in Y \), then \( M_i = \emptyset \) for all \( i\geq 3 \), 
  and \( P_2, P_3, \dots, P_w \) are nonintersecting.

   We will further simplify the right-hand side of~\eqref{eq:lemMUD*LGV} by
  finding a sign-reversing involution. Let \( \mathbf{Q}\in Y \) with
  \( Q_i=(P_i,M_i) \), and let \( \mathbf{P}=(P_1, P_2, \dots, P_w) \). Find the smallest
  \( j \) satisfying one of the following conditions:
\begin{description}
\item[\sc Case 1] \( j\in M_1 \) and \( u=(2,2j-1)\in P_1 \cap P_2 \). In this case, let
  \( v=(t+1, 2j-1) \), where \( t\ge2 \) is the smallest integer such
  that \( (t+1,2j-1) \) is not contained in any path in
  \( \mathbf{P} \). Let \( \mathbf{Q}' \) be the family obtained from
  \( \mathbf{Q} \) by removing the integer \( j \) from \( M_1 \),
  exchanging all points \( (x,y) \) of \( P_1 \) and \( P_2 \) with
  \( y\ge 2j \), and replacing \( u\in P_1 \) and
  \( (t,2j-1) \in P_t \) by \( (1,2j-1) \) and \( v \), respectively.
\item[\sc Case 2] \( j\notin M_1 \), \( P_i \) visits
  \( (i,2j-2), (i, 2j-1), (i,2j) \) for \( i=1,2,\dots, t-1 \), and
  \( P_t \) visits \( (t, 2j-2),\break (t+1, 2j-1), (t, 2j) \) for some
  \( t\geq 2 \). In this case, let \( u=(2,2j-1) \) and
  \( v=(t+1, 2j-1) \). Let \( \mathbf{Q}' \) be the family obtained
  from \( \mathbf{Q} \) by adding the integer \( j \) to \( M_1 \),
  exchanging all points \( (x,y) \) of \( P_1 \) and \( P_2 \) with
  \( y\ge 2j \), and replacing \( (1,2j-1)\in P_1 \) and
  \( v \in P_t \) by \( u \) and \( (t,2j-1) \), respectively.
\end{description}
See Figure~\ref{fig:image3}. If there is no such \( j \), define \( \mathbf{Q}' = \mathbf{Q} \). 
Then the map \(  \mathbf{Q}  \mapsto \mathbf{Q}'  \) is a
sign-reversing involution on \( Y \).
Note that, if \( \mathbf{Q}\in Y \) is a fixed point, then \( M_i=\emptyset \) for all \( i\geq 2 \). 
Hence, a fixed point \( \mathbf{Q}\in Y \) has no intersection point and can be identified with a pair 
\( (T,S) \in  \MUD_n^*(w; (1^k,0^{w-k}) \to (0^w)) \), where \( S=M_1 \). 
This shows the first identity, Equation~\eqref{eq:1}.

The second identity, Equation~\eqref{eq:2}, is easily obtained by combining 
the application of the Lindstr\"om--Gessel--Viennot lemma
\cite[Lemma~1]{LindAA} to~\eqref{eq:UD-lem3}
with the fact that \( e(\vx_n) \overline{e}(\vx_n) = \overline{e}(\vx_n^2) \). 
 \end{proof}

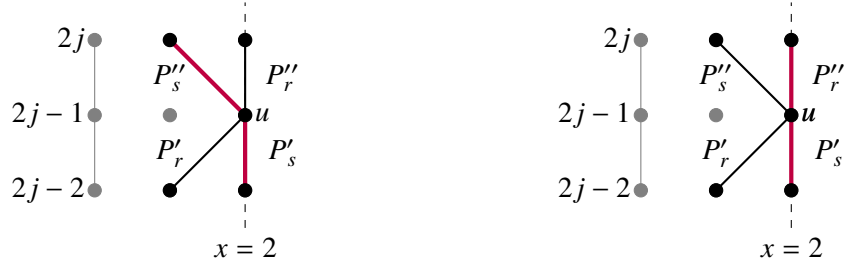
\begin{figure}
\begin{tikzpicture}

\draw[dashed] (2,-.5) -- (2,2.5);
\draw[color=gray] (0,0) -- (0,2);
\foreach \i in {0,1,2}
\foreach \j in {0,1,2}
\filldraw[fill=gray, color=gray] (\i,\j) circle (2.5pt);

\node[left] at (0,0) {\( 2j-2 \)};
\node[left] at (0,1) {\( 2j-1 \)};
\node[left] at (0,2) {\( 2j \)};
\node[below] at (2,-0.5) {\( x=2 \)};
\node[right] at (2,1) {\( u \)};
\node at (1,0.5) {\( P'_r \)};
\node at (2.5,0.5) {\( P'_s \)};
\node at (1,1.5) {\( P''_s \)};
\node at (2.5,1.5) {\( P''_r \)};

\draw[ultra thick, color=purple] 
(2,0) -- (2,1) -- (1,2) ;

\foreach \i in {0,1,2}
\filldraw (2,\i) circle (2.5pt);

\filldraw (1,0) circle (2.5pt);
\filldraw (1,2) circle (2.5pt);

\draw[thick] 
(1,0) -- (2,1) -- (2,2) ;

\end{tikzpicture}
\qquad \qquad \qquad \qquad
\begin{tikzpicture}

\draw[dashed] (2,-.5) -- (2,2.5);
\draw[color=gray] (0,0) -- (0,2);
\foreach \i in {0,1,2}
\foreach \j in {0,1,2}
\filldraw[fill=gray, color=gray] (\i,\j) circle (2.5pt);

\node[left] at (0,0) {\( 2j-2 \)};
\node[left] at (0,1) {\( 2j-1 \)};
\node[left] at (0,2) {\( 2j \)};
\node[below] at (2,-0.5) {\( x=2 \)};
\node[right] at (2,1) {\( u \)};
\node[right] at (2,1) {\( u \)};
\node at (1,0.5) {\( P'_r \)};
\node at (2.5,0.5) {\( P'_s \)};
\node at (1,1.5) {\( P''_s \)};
\node at (2.5,1.5) {\( P''_r \)};

\draw[ultra thick, color=purple] 
(2,0) -- (2,1) -- (2,2) ;

\foreach \i in {0,1,2}
\filldraw (2,\i) circle (2.5pt);

\filldraw (1,0) circle (2.5pt);
\filldraw (1,2) circle (2.5pt);

\draw[thick] 
(1,0) -- (2,1) -- (1,2) ;

\end{tikzpicture}

\caption{Suppose that the black path \( P_r=P'_rP''_r \) has a marking
  \( j\in M_r \). It intersects with the red path \( P_s=P'_sP''_s \)
  at \( u \) but this cannot be canceled because the configuration (on
  the right) obtained by exchanging the two subpaths after \( u \) is
  not a valid configuration due to the marking \( j\in M_r \).}
\label{fig:image2}
\end{figure}

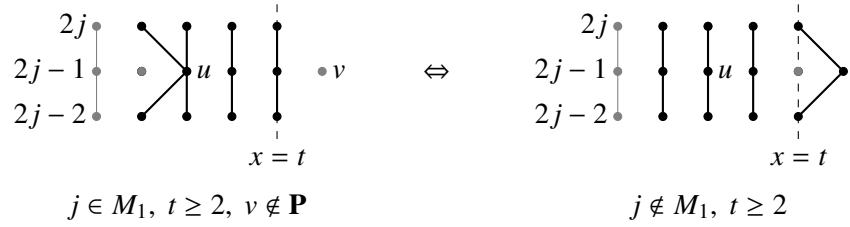
\begin{figure}
\begin{tikzpicture}[scale=0.6]

\draw[color=gray] (0,0) -- (0,2);
\foreach \i in {0,1,2}
\foreach \j in {0,1,2}
\filldraw[fill=gray, color=gray] (\i,\j) circle (2.5pt);

\node[left] at (0,0) {\( 2j-2 \)};
\node[left] at (0,1) {\( 2j-1 \)};
\node[left] at (0,2) {\( 2j \)};
\node[below] at (4,-0.5) {\( x=t \)};
\node[right] at (2,1) {\( u \)};
\node[right] at (5,1) {\( v \)};
\node[below] at (2,-1.5) {\( j \in M_1, ~t\geq 2,~ v\not\in \mathbf{P}\) };
\node[right] at (7,1) {\( \Leftrightarrow \)};

\foreach \i in {1,2,3,4}
\foreach \j in {0,1,2}
\filldraw (\i,\j) circle (2.5pt);

\filldraw[fill=gray, color=gray] (5,1) circle (2.5pt);
\filldraw[fill=gray, color=gray] (1,1) circle (2.5pt);

\foreach \i in {2,3,4}
\draw[thick] (\i,0) -- (\i,1) -- (\i,2);

\draw[thick] (1,0) -- (2,1) -- (1,2);

\draw[dashed] (4,-.5) -- (4,2.5);
\end{tikzpicture}
\qquad  
\begin{tikzpicture}[scale=0.6]

\draw[dashed] (4,-.5) -- (4,2.5);
\draw[color=gray] (0,0) -- (0,2);
\foreach \i in {0,1,2}
\foreach \j in {0,1,2}
\filldraw[fill=gray, color=gray] (\i,\j) circle (2.5pt);

\node[left] at (0,0) {\( 2j-2 \)};
\node[left] at (0,1) {\( 2j-1 \)};
\node[left] at (0,2) {\( 2j \)};
\node[below] at (4,-0.5) {\( x=t \)};
\node[right] at (2,1) {\( u \)};
\node[below] at (2,-1.5) {\( j \not\in M_1, ~t\geq 2 \) };

\foreach \i in {1,2,3,4}
\foreach \j in {0,1,2}
\filldraw (\i,\j) circle (2.5pt);

\filldraw[fill=gray, color=gray] (4,1) circle (2.5pt);
\filldraw (5,1) circle (2.5pt);

\foreach \i in {1,2,3}
\draw[thick] (\i,0) -- (\i,1) -- (\i,2);

\draw[thick] (4,0) -- (5,1) -- (4,2);

\end{tikzpicture}

\caption{An illustration of the involution in Lemma~\ref{lem:5}.}
\label{fig:image3}
\end{figure}

\begin{defn}\label{def:mud_e} 
  Let \( p(T) \) denote the number of length-peaks in \( T \).
  We define
 \begin{align*}
  \MUD^{e,e}_n(w;\mu\to \nu) &= \MUD^1_n(w;\mu\to \nu) \cup \MUD^{0,e}_n(w;\mu\to \nu),\\
  \MUD^{e,o}_n(w;\mu\to \nu) &= \MUD^1_n(w;\mu\to \nu) \cup \MUD^{0,o}_n(w;\mu\to \nu),
  \end{align*}
where 
\begin{align*}
  \MUD^1_n(w;\mu\to \nu)
  &= \{(T,S): T\in \widehat{\UD}_n(w;\mu\to \nu), E_w(T)\neq \emptyset, \mbox{ and } S \subseteq E_w(T) \setminus \{\min E_w(T)\} \}, \\
  \MUD^{0,e}_n(w;\mu\to \nu)  &= \{(T,\emptyset): T\in \widehat{\UD}_n(w;\mu\to \nu), E_w(T)= \emptyset, \mbox{ and \( p(T) \) is even}\}, \\
  \MUD^{0,o}_n(w;\mu\to \nu)  &= \{(T,\emptyset): T\in \widehat{\UD}_n(w;\mu\to \nu), E_w(T)= \emptyset, \mbox{ and \( p(T) \) is odd}\}.
\end{align*}
\end{defn}

Note that, if \( (T,S) \) is an element in \( \MUD^{0,e}_n(w;\mu\to \nu) \)
or \( \MUD^{0,o}_n(w;\mu\to \nu) \), then \( S=\emptyset \).

We are ready to state our combinatorial interpretations of
the sums of Schur functions over partitions with an even bound on the
number of columns and a given number of odd columns appearing
in~\eqref{eq:schur-even-k}
and~\eqref{eq:schur-even-2h-k}. 

\begin{thm}\label{thm:UD-tab}
  For \( 0\le k \leq w-1 \), we have
  \begin{align}
    \label{eq:UD-even}
    \underset {c(\lambda) = k}
    {\sum_{\lambda:\lambda_1 \le 2w}} s_{\lambda}(\vx_n)
     &= \sum_{(T,S)\in \MUD^{e,e}_n(w;(1^k,0^{w-k})\to (0^w))} \omega(T),\\
    \label{eq:UD-even'}
    \underset {c(\lambda) = 2w-k}
    {\sum_{\lambda:\lambda_1 \le 2w}} s_{\lambda}(\vx_n)
      &= \sum_{(T,S)\in \MUD^{e,o}_n(w;(1^k,0^{w-k})\to (0^w))} \omega(T).
  \end{align}
 For the case \( k=w \), we have
 \begin{align}
 \label{eq:UD-even-h}
 \underset {c(\lambda) = w}
 {\sum_{\lambda:\lambda_1 \le 2w}} s_{\lambda}(\vx_n)
     &= \sum_{(T,S)\in \MUD^*_n(w;(1^w)\to (0^w))} \omega(T).
  \end{align}
\end{thm}

\begin{proof}
  The third identity, Equation~\eqref{eq:UD-even-h}, follows immediately
  from~\eqref{eq:schur-even-k} and~\eqref{eq:1}. We will only prove the
  first identity, Equation~\eqref{eq:UD-even}, because
  \eqref{eq:UD-even'} can be 
  proved similarly.

  From now on we assume that \( 0 \leq k \leq w-1 \) and let
  \( \mu=(1^k,0^{w-k}) \) and \( \nu=(0^w) \). By adding the two
  equations~\eqref{eq:schur-even-k} and~\eqref{eq:schur-even-2h-k}
  and subsequently 
  applying~\eqref{eq:1} and~\eqref{eq:2}, we arrive at
\begin{equation}\label{eq:4}
  \underset {c(\lambda) = k}
  {\sum_{\lambda:\lambda_1 \le 2w}} s_{\lambda}(\vx_n)
  = \frac{1}{2}
  \left(  \sum_{(T,S)\in \MUD_n^*(w; \mu\to \nu)} \omega(T)
  + \sum_{(T,S)\in \MUD^<_n(w; \mu\to \nu)} \omega(T) \prod_{j\in S}(-x_j^2)
 \right).
\end{equation}  

We divide \( \MUD^*_n(w;\mu\to \nu) = A\cup B \) into two subsets, where
\begin{align}
  \notag
A &= \{(T,S) \in \MUD^*_n(w;\mu\to \nu): E_w(T)\ne \emptyset\},\\
  \label{eq:7}
B &= \{(T,S) \in \MUD^*_n(w;\mu\to \nu): E_w(T)= \emptyset\}.
\end{align}
Then \( A \) is the disjoint union of the following two sets:
\begin{align*}
  \MUD^1_n(w;\mu\to \nu)
  &= \{(T,S) \in \MUD^*_n(w;\mu\to \nu): E_w(T)\neq \emptyset \mbox{ and } \min E_w(T)\not \in S\}, \\
  A'  &= \{(T,S) \in \MUD^*_n(w;\mu\to \nu): E_w(T)\neq \emptyset \mbox{ and } \min E_w(T) \in S\}.
\end{align*}
Since \( (T,S)\mapsto (T,S \cup \{\min E_w(T)\}) \) is a
weight-preserving bijection from
\( \MUD^1_n(w;\mu\to \nu) \) to \( A' \), we obtain
\begin{equation}\label{eq:3}
\frac{1}{2}  \sum_{(T,S)\in A}  \omega(T)  = \sum_{(T,S)\in \MUD^1_n(w; \mu\to \nu)} \omega(T).
\end{equation}

Now we claim that
\begin{equation}\label{eq:5}
  \frac{1}{2}
  \left( \sum_{(T,S)\in B}  \omega(T)  +
\sum_{(T,S)\in \MUD^<_n(w; \mu\to \nu)} \omega(T) \prod_{j\in S}(-x_j^2) \right)
   = \sum_{(T,S)\in \MUD^{0,e}_n(w;\mu\to \nu)} \omega(T).
\end{equation}
Note that the first identity, Equation~\eqref{eq:UD-even}, follows
from~\eqref{eq:4}, \eqref{eq:3}, and~\eqref{eq:5}. Thus it remains to prove
the above claim.

By \eqref{eq:MUD*} and \eqref{eq:7},
we have \( S=\emptyset \) for every \( (T,S)\in B \). Thus, we can rewrite \( B \) as
\[
 B = \{(T,\emptyset): T\in \widehat{\UD}_n(w;\mu\to \nu), E_w(T)= \emptyset\}.
\]
Let \( L \) be the left-hand side of \eqref{eq:5}. Then
\begin{equation}\label{eq:13}
 L = \sum_{(T,\emptyset)\in X}  \omega(T)
  + \frac{1}{2} \sum_{(T,\emptyset)\in Y} \omega(T) 
  + \frac{1}{2} \sum_{(T,S)\in Z} \omega(T) \prod_{j\in S}(-x_j^2),
\end{equation}
where
\begin{align*}
  X &=B\cap \MUD^<_n(w;\mu\to \nu),\\
  Y &=B\setminus \MUD^<_n(w;\mu\to \nu),\\
  Z &= \MUD^<_n(w;\mu\to \nu) \setminus B.
\end{align*}

We need to find equivalent descriptions for the sets \( X \), \( Y \),
and \( Z \). Suppose \( (T, \emptyset)\in X \). Since \( (T, \emptyset)\in B \), we have 
\( E_w(T) = \emptyset \), and \( T \) has no
non-full length-peaks. Moreover, since \( (T, \emptyset) \in \MUD^<_n(w;\mu\to \nu) \), we
have \( \ell(T_j)<w \) for all \( j \), which implies that \( T \) has no
full length-peaks. Conversely, any pair \( (T,\emptyset) \) satisfying
these properties is an element of \( X \) because
\( E_w(T)= \emptyset \) and \( p(T)=0 \) imply that \( \ell(T_j)<w \)
for all \( j \). This shows that
\begin{equation}\label{eq:10}
  X= \{(T,\emptyset): T\in \widehat{\UD}_n(w;\mu\to \nu), E_w(T)= \emptyset, p(T)=0 \}.
\end{equation}
By similar arguments, we have
\begin{align}\label{eq:11}
   Y&= \{(T,\emptyset): T\in \widehat{\UD}_n(w;\mu\to \nu), E_w(T)= \emptyset, p(T)\ne0 \},\\
  \label{eq:12}
  Z &= \{(T,S) \in \MUD^<_n(w;\mu\to \nu): \mbox{\( p(T)\ne0 \) or \( S\ne \emptyset\)}  \}.
\end{align}

Now we simplify the sum over \( Z \) in \eqref{eq:13} by finding a
sign-reversing involution on \( Z \). Suppose \( (T,S)\in Z \).
Let \( j \) be the smallest integer satisfying one of the following conditions:
\begin{itemize}
\item \( j\in S \) and \( \ell(T_{2j-1}) < w-1 \);
\item \( j\not\in S \) and \( 2j-1 \) is a length-peak of \( T \).
\end{itemize}
If there is such \( j \), we define \( \phi(T,S) = (T',S') \) as shown
in Figure~\ref{fig:image9}. Otherwise, let \( \phi(T,S) = (T,S) \).
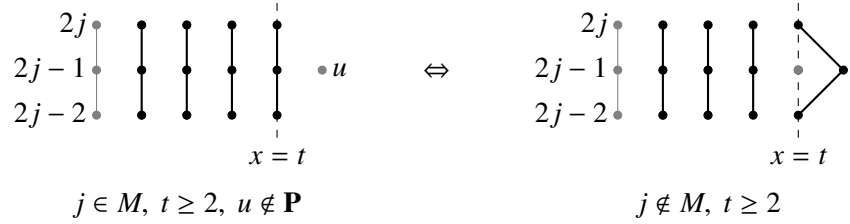
\begin{figure}
\begin{tikzpicture}[scale=0.6]

\draw[color=gray] (0,0) -- (0,2);
\foreach \i in {0,1,2}
\foreach \j in {0,1,2}
\filldraw[fill=gray, color=gray] (\i,\j) circle (2.5pt);

\node[left] at (0,0) {\( 2j-2 \)};
\node[left] at (0,1) {\( 2j-1 \)};
\node[left] at (0,2) {\( 2j \)};
\node[below] at (4,-0.5) {\( x=t \)};
\node[right] at (5,1) {\( u \)};
\node[below] at (2,-1.5) {\( j \in M, ~t\geq 2,~ u\not\in \mathbf{P}\) };
\node[right] at (7,1) {\( \Leftrightarrow \)};

\foreach \i in {1,2,3,4}
\foreach \j in {0,1,2}
\filldraw (\i,\j) circle (2.5pt);

\filldraw[fill=gray, color=gray] (5,1) circle (2.5pt);

\foreach \i in {1,2,3,4}
\draw[thick] (\i,0) -- (\i,1) -- (\i,2);

\draw[dashed] (4,-.5) -- (4,2.5);
\end{tikzpicture}
\qquad  
\begin{tikzpicture}[scale=0.6]

\draw[dashed] (4,-.5) -- (4,2.5);
\draw[color=gray] (0,0) -- (0,2);
\foreach \i in {0,1,2}
\foreach \j in {0,1,2}
\filldraw[fill=gray, color=gray] (\i,\j) circle (2.5pt);

\node[left] at (0,0) {\( 2j-2 \)};
\node[left] at (0,1) {\( 2j-1 \)};
\node[left] at (0,2) {\( 2j \)};
\node[below] at (4,-0.5) {\( x=t \)};
\node[below] at (2,-1.5) {\( j \not\in M, ~t\geq 2 \) };

\foreach \i in {1,2,3,4}
\foreach \j in {0,1,2}
\filldraw (\i,\j) circle (2.5pt);

\filldraw[fill=gray, color=gray] (4,1) circle (2.5pt);
\filldraw (5,1) circle (2.5pt);

\foreach \i in {1,2,3}
\draw[thick] (\i,0) -- (\i,1) -- (\i,2);

\draw[thick] (4,0) -- (5,1) -- (4,2);

\end{tikzpicture}

\caption{Configurations for the involution \( \phi \).}
\label{fig:image9}
\end{figure}
Then \( \phi \) is a sign-reversing involution on \( Z \) whose fixed
point set \( Z_0 \) is the set of pairs
\( (T,S)\in \MUD^<_n(w;\mu\to \nu) \) satisfying the following conditions:
\begin{itemize}
\item \( S\ne \emptyset\);
\item if \( j\in S \), then
\( \ell(T_{2j-1}) = w-1 \);
\item if \( j\not\in S \), then \( 2j-1 \) is
not a length-peak of \( T \). 
\end{itemize}

Now, for \( (T,S)\in Z_0 \),
let \( \psi(T,S) = (T',\emptyset) \),
where \( T' \) is the \( w \)-up-down tableau
obtained from \( T \) by applying the operation
in Figure~\ref{fig:image10} for each \( j\in S \).

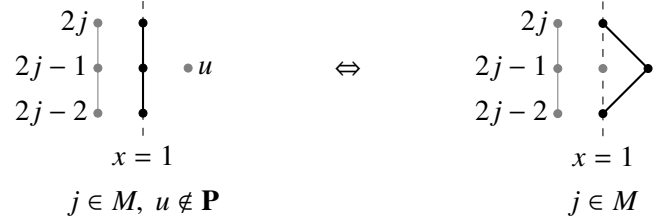
\begin{figure}
\begin{tikzpicture}[scale=0.6]

\draw[color=gray] (0,0) -- (0,2);
\foreach \i in {0}
\foreach \j in {0,1,2}
\filldraw[fill=gray, color=gray] (\i,\j) circle (2.5pt);

\node[left] at (0,0) {\( 2j-2 \)};
\node[left] at (0,1) {\( 2j-1 \)};
\node[left] at (0,2) {\( 2j \)};
\node[below] at (1,-0.5) {\( x=1 \)};
\node[right] at (2,1) {\( u \)};
\node[below] at (1,-1.5) {\( j \in M, ~ u\not\in \mathbf{P}\) };
\node[right] at (5,1) {\( \Leftrightarrow \)};

\foreach \i in {1}
\foreach \j in {0,1,2}
\filldraw (\i,\j) circle (2.5pt);

\filldraw[fill=gray, color=gray] (2,1) circle (2.5pt);

\draw[thick] (1,0) -- (1,1) -- (1,2);

\draw[dashed] (1,-.5) -- (1,2.5);
\end{tikzpicture}
\qquad  \quad
\begin{tikzpicture}[scale=0.6]

\draw[dashed] (1,-.5) -- (1,2.5);
\draw[color=gray] (0,0) -- (0,2);
\foreach \i in {0}
\foreach \j in {0,1,2}
\filldraw[fill=gray, color=gray] (\i,\j) circle (2.5pt);

\node[left] at (0,0) {\( 2j-2 \)};
\node[left] at (0,1) {\( 2j-1 \)};
\node[left] at (0,2) {\( 2j \)};
\node[below] at (1,-0.5) {\( x=1 \)};
\node[below] at (1,-1.5) {\( j \in M\) };

\foreach \i/\j in {1/0,2/1,1/2}
\filldraw (\i,\j) circle (2.5pt);

\filldraw[fill=gray, color=gray] (1,1) circle (2.5pt);

\draw[thick] (1,0) -- (2,1) -- (1,2);

\end{tikzpicture}

\caption{Configurations for the bijection \( \psi \).}
\label{fig:image10}
\end{figure}

It is easy to see that \( \psi \) is a bijection from \( Z_0 \) to
\( Y \) such that, if \( \psi(T,S) = (T',\emptyset) \), then
\( \omega(T) \prod_{j\in S}(-x_j^2) = \omega(T')(-1)^{p(T')} \). This
shows that
\[
\frac{1}{2} \sum_{(T,S)\in Z} \omega(T) \prod_{j\in S}(-x_j^2)
 = \frac{1}{2} \sum_{(T,S)\in Z_0} \omega(T) \prod_{j\in S}(-x_j^2)
 = \frac{1}{2} \sum_{(T,\emptyset)\in Y} \omega(T) (-1)^{p(T)}.
\]
Thus \eqref{eq:13} can be rewritten as
\[
L = \sum_{(T,\emptyset)\in X}  \omega(T)
  + \sum_{(T,\emptyset)\in Y_1} \omega(T) ,
\]
where \( Y_1 \) is the set of \( (T,\emptyset) \) such that
\( T\in \widehat{\UD}_n(w;\mu\to \nu) \), \( E_w(T)= \emptyset \), and
\( p(T) \) is a nonzero even integer. Then, by~\eqref{eq:10},
\( X\cup Y_1 \) is the set of \( (T,\emptyset) \) such that
\( T\in \widehat{\UD}_n(w;\mu\to \nu) \), \( E_w(T)= \emptyset \), and
\( p(T) \) is even, which is exactly the set
\( \MUD^{0,e}_n(w;\mu\to \nu) \).
Therefore
\[
   L = \sum_{(T,S)\in \MUD^{0,e}_n(w;\mu\to \nu)} \omega(T),
\]
and we obtain the claim \eqref{eq:5}.
This completes the proof of~\eqref{eq:UD-even}.
\end{proof}

\section{Standard Young tableaux and lattice walks}\label{sec:SYT_walks}

In this section, we show that the number of standard Young tableaux of
bounded width with given number of odd columns is equal to the
number of certain lattice walks.

We begin by introducing notation for the relevant sets of standard
Young tableaux and of vacillating tableaux.

\begin{defn}\label{def:1}
Let \( \SYT_{n,w}[k] \) denote the set of
standard Young tableaux in \( \SYT_{{n},w} \) having exactly \( k \)
columns of odd length.
\end{defn}

The following definition introduces the lattice paths that will
feature in our results. We rather prefer to use the language of {\it
vacillating tableaux}, although everything could equivalently be
formulated in terms of lattice paths.

\begin{defn}
  A \emph{\( w \)-vacillating tableau} of length \( n \) is a sequence
  \( T=( T_0,  T_1, \dots,  T_n) \) of partitions with\break \( \ell( T_i) \leq w \),
where the partitions \(  T_{i-1} \) and \(  T_{i} \) differ by at most one cell for \( i=1,2,\dots, n \).
We denote by \( \VT_n(w;\mu\to \nu) \) the set of \( w \)-vacillating tableaux \( ( T_0,  T_1, \dots,  T_n) \)
satisfying \(  T_0=\mu \) and \(  T_n=\nu \).
We define \( \VT^>_n(w;\mu\to \nu) \) to be 
the subset of \( w \)-vacillating tableaux in \( \VT_n(w;\mu\to \nu) \) 
with the property that equality of~$ T_{i-1}$ and~$ T_i$
can only occur when $ \ell(T_{i-1})=\ell(T_i)=w$.
\end{defn}

Suppose that \( T=( T_0,  T_1, \dots,  T_n) \in \VT_n(w;\mu\to \nu) \). 
Note that each partition \(  T_i \) can be considered as a \( w \)-tuple of non-increasing integers.
Hence, by definition, \( T \) can be identified with a walk of length \( n \)
from \( \mu \) to \( \nu \) in the region \( \{(x_1,\dots,x_w):x_1\ge \dots\ge x_w\ge0\} \) 
using steps in \( \{ \vec0, \pm \epsilon_1,\dots,\pm \epsilon_w\} \), where \( \epsilon_i \) is the \( i \)-th standard basis vector.

\medskip
Our first result concerns 
standard Young tableaux with an odd bound on the width and a given
number of odd-length columns.
Note that if \( n \not\equiv k \mod 2 \), then \( |\SYT_{{n},2w+1}[k]|=0 \). 
By taking the coefficients of \( x_1x_2 \cdots x_n \) on both sides of~\eqref{eq:UD-odd1} and~\eqref{eq:UD-odd2}, we obtain the following theorem,
which was first proved by the fourth author in \cite[Corollary~5.5(3)]{Okada2016} 
by considering the \( n \)-th tensor power of the natural representation 
of the Lie algebra \( \mathfrak{so}_{2w+1} \).

\begin{thm}\label{thm:SYTodd}
  For nonnegative integers \( n \), \( w \) and \( k \) such that \( 0 \leq k \leq 2w+1 \) and  \( n \equiv k \mod 2 \), we have
  \[ 
  |\SYT_{{n},2w+1}[k]| = |\VT^>_n(w;(1^t,0^{w-t})\rightarrow (0^w) )|,
  \]
  where \( t = k \) if \( k\le w \) and \( t = 2w+1-k \) if \( k>w \). 
  \end{thm}
  
 \begin{proof}
   First, suppose \( 0\le k\le w \). Extracting the coefficient of
   \( x_1 \cdots x_n \) on both sides of~\eqref{eq:UD-odd1}, 
   we obtain that \( |\SYT_{{n},2w+1}[k]| \) is equal to
   the number of pairs
   \( (T,S) \in \MUD_n^o(w; (1^k,0^{w-k})\rightarrow (0^w) )\) such
   that \( \omega(T)\omega(S) =x_1x_2\cdots x_n \) and \( |S| \) is
   even. By \Cref{def:MUD^o}, this is equal to the number of the pairs
   \( (T,S) \) of a \( w \)-up-down tableau
   \( T=( T_0,  T_1,\dots,  T_{2n}) \) and a set \( S \subseteq [n] \)
   such that, for each \( j\in[n] \), one of the following two
   conditions holds:
 \begin{itemize}
 \item \( j \not\in S \), \(  T_{2j-2} \) and \(  T_{2j} \) differ by one cell, 
 and \(  T_{2j-1} \) is the larger partition between \(  T_{2j-2} \) and \(  T_{2j} \);
 \item \( j \in S \), \(  T_{2j-2}= T_{2j-1}= T_{2j} \), and \( \ell( T_{2j-1})=w \).
 \end{itemize}

 Since the odd-indexed partitions \(  T_{2j-1} \) are redundant and
 \( j \in S \) if and only if \(  T_{2j-2} =  T_{2j} \), we can identify
 these pairs \( (T,S) \) with the walks
 \( ( T_{0},  T_{2}, \dots,  T_{2n})\in \VT_n(w;(1^k,0^{w-k})\rightarrow
 (0^w) ) \) with the property that a zero step can only occur when
 \( x_w>0 \). This shows that
 \( |\SYT_{{n},2w+1}[k]| = |\VT^>_n(w;(1^k,0^{w-k})\rightarrow (0^w)|
 \). The other case \( w+1\le k\le 2w+1 \) can be handled similarly
 using~\eqref{eq:UD-odd2}.
 \end{proof}

 \begin{rem}
It is easy to see that \(  |\VT^>_{n}(1;\vec0 \rightarrow \vec0 )| \) is equal to
the number of Motzkin paths of length \( n \) having no horizontal
step on the \( x \)-axis, which is called {\it Riordan number}.  
The Riordan number also counts the number of noncrossing (respectively~nonnesting) partitions of \( [n] \) without singletons; see \href{https://oeis.org/A005043}{A005043}.
In general,  
\( |\VT^>_n(w; \vec0 \rightarrow \vec0 )| \) equals the multiplicity 
of the trivial representation in the \( n \)-th tensor power of the natural representation 
of \( \mathfrak{so}_{2w+1} \). 
See \cite{Okada2016}.
 \end{rem}

\medskip
Now we turn to numbers of standard Young tableaux
with an even bound on the width and a given
number of odd-length columns.
On the other side of the identities to come we need to consider {\it marked\/}
vacillating tableaux.  

\begin{defn}
  Let \( \MVT_n(w;\mu\to \nu) \) denote the set of pairs
  \( (T,S) \) of \( T \in \VT_n(w;\mu\to \nu) \) and
  \( S\subseteq [n] \). We call each element
  \( (T,S)\in \MVT_n(w;\mu\to \nu) \) a \emph{marked vacillating tableau}.
\end{defn}

\begin{defn}
  Let \( \MVT^*_n(w;\mu\to \nu) \) denote the set of
  elements \( (T, S) \in \MVT_n(w;\mu\to \nu) \) such
  that \( T \) has no zero step and, if \( j \in S \), then the
  \( j \)-th step of \( T \) is \( \epsilon_w \) whose starting point
  is on the hyperplane \( x_w=0 \).

  Let \( \MVT^0_n(w;\mu\to \nu) \) denote the set of elements
  \( (T, S) \in \MVT^*_n(w;\mu\to \nu) \) satisfying one of the following
  conditions:
  \begin{itemize}
  \item \( T \) has no \( \epsilon_w \) steps, or
  \item \( T \) has at least one \( \epsilon_w \) step and
    \( j \not\in S \), where the \( j \)-th step of \( T \) is the
    first \( \epsilon_w \) step in \( T \).
  \end{itemize}

  Let \( \MVT^1_n(w;\mu\to \nu) \) denote the set of elements
  \( (T, S) \in \MVT^*_n(w;\mu\to \nu) \) such that \( T \) has at
  least one \( \epsilon_w \) step and \( j \in S \), where the
  \( j \)-th step of \( T \) is the first \( \epsilon_w \) step in
  \( T \).
\end{defn}

Note that
\( \MVT^*_n(w;\mu\to \nu) = \MVT^0_n(w;\mu\to \nu) \sqcup
\MVT^1_n(w;\mu\to \nu) \).

Here is our result on equality of numbers of standard Young tableaux
with an even bound on the number of columns and a fixed number of odd-length
columns and numbers of marked vacillating tableaux.
Note that, if \( n \not\equiv k \mod 2 \), then
\( |\SYT_{{n},2w}[k]|=0 \).
The third equation of~(1) and the equality in~(2) below already appeared 
in \cite[Corollary~5.5 (4)]{Okada2016}.

\begin{thm}\label{thm:SYTeven}
 Let \( n \), \( w \), and \( k \) be nonnegative integers such that \( 0 \leq k \leq w \). 
 \begin{enumerate} 
 \item For \( 0 \leq k \leq w-1 \), we have
   \begin{align*}
    |\SYT_{{n},2w}[k]| &= |\MVT^0_n(w;(1^k,0^{w-k})\rightarrow (0^w))|, \\
    |\SYT_{{n},2w}[2w-k]| &= |\MVT^1_n(w;(1^k,0^{w-k})\rightarrow (0^w))|,\\
  |\SYT_{{n},2w}[k]|+|\SYT_{n,2w}[2w-k]|&=|\MVT^{*}_n(w;(1^k,0^{w-k})\rightarrow (0^w) )|. 
   \end{align*}
  \item For \( k=w \), we have \( |\SYT_{{n},2w}[w]|=  |\MVT^{*}_n(w;(1^w)\rightarrow (0^w) )| \). 
 \end{enumerate}
\end{thm}

\begin{proof}
  This can be proved in a similar manner as \Cref{thm:SYTodd} by
  extracting the
  coefficients of \( x_1x_2\cdots x_n \) in the equations in
  \Cref{thm:UD-tab}, except that in the present case the elements \(
  j \in S \) do not affect the weights~\( x_j \). 
\end{proof}

 \begin{rem}
It is easy to see that \(  |\MVT^{*}_{2n}(1;\vec0 \rightarrow \vec0 )| \) is equal to
the number of Dyck paths of semilength \( n \) with the property
that each up step starting on the \( x \)-axis is allowed to be marked.
The number of these marked Dyck paths is given by
the central binomial coefficients; see \href{https://oeis.org/A000984}{A000984}.
Furthermore,
we have \( |\MVT^{*}_{2n}(2; \vec0 \rightarrow \vec0)| = ( \Cat(n) )^2 \), 
where \( \Cat(n) \) is the \( n \)-th Catalan number; 
see \href{https://oeis.org/A001246}{A001246}.
This follows from the following three claims:

\medskip\noindent
(1)
\( |\MVT^{*}_{2n}(w; \vec0 \rightarrow \vec0)| \) is the multiplicity 
of the trivial representation in the \( (2n) \)-th tensor power of the natural representation 
of \( \mathfrak{so}_{2w} \).

\smallskip\noindent
(2)
The Lie algebra \( \mathfrak{so}_4 \) is isomorphic to the direct sum 
\( \mathfrak{sl}_2 \oplus \mathfrak{sl}_2 \), and the natural representation of \( \mathfrak{so}_4 \) 
corresponds to the outer tensor product of the natural representations of \( \mathfrak{sl}_2 \).

\smallskip\noindent
(3)
The multiplicity of the trivial representation in the \( (2n) \)-th tensor power of 
the natural representation of \( \mathfrak{sl}_2 \) is equal to the Catalan number \( \Cat(n) \).
\end{rem}

\section{Final questions}
\label{sec:questions}

For \( w\geq 1 \), we have the following result; see \cite[Example~2 on
page~423]{Macdonald} and also~\cite[Eq.~(10.5)]{KLO}.

\begin{prop}\label{prop:Haar}
The number  \( |\SYT_{{n},w}[0]| \) is the expected value of \( \operatorname{trace}(O)^{n} \),
  where \( O \) is a \( w \times w \) orthogonal matrix randomly
  selected according to the Haar measure.
\end{prop}

The result of Proposition~\ref{prop:Haar} raises the following questions.

\begin{quest}
  Is there an interpretation for \( |\SYT_{n,w}[k]| \) using orthogonal matrices for any \( k \)?
\end{quest}

Further questions that our results in Section~\ref{sec:comb-interpr} (and in
Section~\ref{sec:SYT_walks}) naturally suggest are the following.

\begin{quest}
Are there bijective proofs for the identities in
Theorems~\ref{thm:Goulden_MUDodd}, \ref{thm:1},
and~\ref{thm:UD-tab}?\newline
Are there bijective proofs for the identities in
Theorems~\ref{thm:SYTodd} and~\ref{thm:SYTeven}?
\end{quest}

There are strong indications that an approach using Fomin's growth
diagrams (cf.\ \cite{Krattenthaler2016}) should be successful.
As for evidence, we point out that Theorem~5
in~\cite{Krattenthaler2016} gives a bijection between semistandard Young
tableaux with an even bound on the number of {\it rows} and a fixed
number of odd-length columns and certain down-up tableaux.
We are convinced
that this does produce a bijective proof of Goulden's
identity~\eqref{eq:Goulden2m,k}, however after application of the
symmetric function involution~$\omega$ that maps elementary symmetric
functions to complete homogeneous symmetric functions (which is in
fact Goulden's original formulation of the identity). Furthermore,
we believe that,
by modifying the proof in~\cite{Krattenthaler2016} so that instead of
the first and fourth variation of the Robinson--Schensted--Knuth
algorithm we would use the second and third variation of that
algorithm in terms of growth diagrams (we
refer the reader to~\cite[proof of Theorem~5]{Krattenthaler2016}
and~\cite[Sec.~4]{KratCE}), one would obtain a bijective proof of
Theorem~\ref{thm:Goulden_MUDeven}.

\subsection*{Acknowledgements}
J.H.~was supported by the National Research Foundation of Korea (NRF) grant
funded by the Korea government (MSIT) (IRIS RS-2026-25485989).
J.S.K.~was supported by the National Research Foundation of Korea (NRF) grant funded by the Korea government
(MSIT) (IRIS RS-2025-00557835).
C.K.~was partially supported by the Austrian
Science Foundation FWF, grant 10.55776/F1002, in the framework
of the Special Research Program ``Discrete Random Structures:
Enumeration and Scaling Limits".
S.O.~was partially supported by JSPS KAKENHI grants 21K03202 and 24K06646.

\end{document}